\documentclass[11pt,reqno]{amsart}

\usepackage{fullpage, amsmath, amsthm, amstext, amssymb, amsfonts, fancyhdr, graphicx, tikz-network, forest, mathrsfs, array, mathtools, bm, url, subfigure, caption, float, appendix, hyperref, comment, xcolor}

\floatstyle{plaintop}
\restylefloat{table}

\NewDocumentCommand\DownArrow{O{2.0ex} O{black}}{%
   \mathrel{\tikz[baseline] \draw [<->, line width=0.5pt, #2] (0,0) -- ++(0,#1);}
}

\allowdisplaybreaks[1]

\newtheorem{theorem}{Theorem}[section]
\newtheorem{lemma}[theorem]{Lemma}
\newtheorem{corollary}[theorem]{Corollary}

\newtheorem{proposition}[theorem]{Proposition}

\theoremstyle{definition}

\newtheoremstyle{boldremark}
    {\dimexpr\topsep/2\relax} 
    {\dimexpr\topsep/2\relax} 
    {}          
    {}          
    {\bfseries} 
    {.}         
    {.5em}      
    {}          

\theoremstyle{boldremark}
\newtheorem{remark}[theorem]{Remark} 

\DeclareMathOperator{\PF}{\mathrm{PF}}

\DeclareMathOperator{\lel}{\mathrm{lel}}

\DeclareMathOperator{\slev}{\mathrm{slev}}
\DeclareMathOperator{\one}{\mathrm{ones}}

\newcommand{\uR}{\boldsymbol{u}}

\newcommand{\Sym}{\mathfrak{S}}

\newcommand{\PR}{\mathbb P}

\newcommand{\FC}{\mathcal F}
\newcommand{\calt}{\mathcal{T}}

\newcommand{\be}{\begin{equation}}
\newcommand{\ee}{\end{equation}}

\newcommand{\beas}{\begin{eqnarray*}}
\newcommand{\eeas}{\end{eqnarray*}}
\newcommand{\dconv}{\overset{\mathrm{d}}{\to}}

\newcommand{\old}[1]{}
\newcommand{\red}[1]{{\color{red} #1}}
\newcommand{\blue}[1]{{\color{blue} #1}}

\forestset{parent color/.style args={#1}{
    {fill=#1},
    for tree={fill/. wrap pgfmath arg={#1!##1}{1/level()*80},draw=#1!80!black}},
    root color/.style args={#1}{fill={{#1!60!black!25},draw=#1!80!black}}
}

\title{Distribution of new statistics of parking functions and their generalizations}

\author[1]{Stephan Wagner}
\address[Stephan Wagner]{Institute of Discrete Mathematics, TU Graz, Steyrergasse 30, 8010 Graz, Austria \and Department of Mathematics, Uppsala University, Box 480, 751 06 Uppsala, Sweden}
\email{\textcolor{blue}{\href{mailto:stephan.wagner@tugraz.at}{stephan.wagner@tugraz.at}}}
\thanks{S. Wagner was supported by the Swedish research council (VR), grant 2022-04030.}

\author[2]{Catherine H. Yan}
\address[Catherine H. Yan]{Department of Mathematics, Texas A \&~M University, College Station, TX 77843, USA}
\email{\textcolor{blue}{\href{mailto:huafei-yan@tamu.edu}{huafei-yan@tamu.edu}}}
\thanks{C.H.~Yan was supported  by the Simons Collaboration Grant for Mathematics 704276.}

\author[3]{Mei Yin}
\address[Mei Yin]{Department of Mathematics, University of Denver, Denver, CO 80208, USA}
\email{\textcolor{blue}{\href{mailto:mei.yin@du.edu}{mei.yin@du.edu}}}
\thanks{M.~Yin was supported by the Simons Foundation Grant MPS-TSM-00007227.}

\thanks{An extended abstract for this paper previously appeared in the Proceedings of the 35th International Conference on Probabilistic, Combinatorial and Asymptotic Methods for the Analysis of Algorithms (AofA 2024) \cite{wagner_et_al:LIPIcs.AofA.2024.29}.}

\begin{document}



\begin{abstract}
In this paper we present new results on the enumeration of parking
functions and labeled forests. We introduce new statistics on parking functions, which are then extended to labeled forests via bijective correspondences. We determine the joint distribution of two statistics on parking functions and their counterparts on labeled forests. Our results on labeled forests also serve to explain the mysterious equidistribution between two seemingly unrelated statistics in parking functions recently identified by Stanley and Yin and give an explicit bijection between the two statistics. Extensions of our techniques are discussed, including joint distribution on further refinement of these new statistics. 
\end{abstract}

\maketitle

\section{Introduction}
\label{sec:intro}

In this paper we present new results on the enumeration of parking
functions and labeled forests. We introduce new statistics on parking functions, which are then extended to labeled forests via bijective correspondences. We determine the joint distribution of two statistics on parking functions and their counterparts on labeled forests. 
Our enumerative results for parking functions and for labeled forests inform each other. In particular, our results on labeled forests serve to explain the mysterious symmetric joint distribution of two seemingly unrelated statistics in parking functions recently identified by Stanley and Yin \cite{SY} and give an explicit bijection between the two statistics.

We begin with the necessary definitions. In the parking
function scenario due to Konheim and Weiss \cite{k-w}, there are $n$
parking spaces on a one-way street, labeled $1,2,\dots,n$ in
consecutive order. A line of $m \leq n$ cars enters the street, one by one. The $i$-th car drives to its preferred spot $\pi_i$ and parks there if possible; if the spot is already occupied then the car parks in the first available spot after that. The list of preferences $\pi=(\pi_1,\dots,\pi_m)$ is called a \emph{parking function} if all cars successfully park. (The parking function is called `classical' when $m=n$.) We denote the set of parking functions by $\PF(m, n)$, where $m$ is the number of cars and $n$ is the number of parking spots. 
Using the pigeonhole principle, we see that a parking function $\pi \in \PF(m, n)$ must have at most one value $=n$, at most two values $\geq n-1$, and for each $k$ at most $k$ values $\ge n-k+1$, and any such function is a parking function. Equivalently, $\pi$ is a parking function if and only if
\begin{equation}\label{pigeon}
\#\{k: \pi_k \leq i\} \geq m-n+i, \hspace{.2cm} \text{for } i=n-m+1, \dots, n.
\end{equation}
We make two immediate observations from \eqref{pigeon}. The first observation is that parking functions are invariant under the action of the symmetric group $\Sym_m$ permuting the $m$ cars, that is, permuting the list of preferences $\pi$. The second observation is that when some $\pi_i$ takes values in the set $\{1, 2, \dots, n-m+1\}$, changing $\pi_i$ to any other value in the set $\{1, 2, \dots, n-m+1\}$ has no effect on $\pi$ being a parking function.

One of the most fundamental results on parking functions is that the number of parking functions is
$\vert \PF(m, n)\vert=(n-m+1)(n+1)^{m-1}$. A famous combinatorial proof in the classical case was given by
Pollak (unpublished but recounted in \cite{Pollak} and
\cite{Pollak2}). See also Pitman and Stanley \cite{PS1} for a generalization of Pollak's circle argument. The combinatorial argument boils down to the following easily verified
statement: Let $G$ denote the group of all $m$-tuples
$(\pi_1,\dots,\pi_m)\in [n+1]^m$ with componentwise addition modulo
$n+1$. Let $H$ be the subgroup of $G$ generated by $(1,1,\dots,1)$. Then
every coset of $H$ contains exactly $n-m+1$ parking functions. Interpreted probabilistically, the combinatorial operation involves assigning $m$ cars on a circle with $n+1$ spots and recording those car assignments where spot $n+1$ is left empty after circular rotation. Since there are $n-m+1$ missing spaces for the
assignment of any preference sequence, any preference sequence $\pi$ has $n-m+1$ rotations that
are valid parking functions. Our parking function proofs will be based on refinements of Pollak's argument, where we investigate the individual parking statistics for each car the moment it is parked on the circle.

This new approach first appeared in a paper by Stanley and Yin \cite{SY} and is extended in this paper, where we introduce the new statistics \emph{leading elements}
and \emph{size of level set} on parking functions $\pi \in \PF(m, n)$ and examine their joint distributions.
The statistic size of level set counts the total number of cars whose preferred spot is in the range $\{1, 2, \dots, n-m+1\}$ and generalizes the $1$'s statistic for classical parking functions, which counts the total number of cars that prefer spot $1$. The statistic leading elements was introduced in \cite{SY} earlier for classical parking functions $\pi \in \PF(n, n)$ and counts the total number of cars whose desired spot is the same as that of the first car. It was shown in \cite[Theorem 4.2]{SY} via a generating function approach that for classical parking functions the leading elements statistic is equidistributed with the widely-studied $1$'s statistic. This feature of parking functions is quite mysterious as these two parking function statistics seem unrelated and are not of the same nature. While the leading elements statistic is invariant under circular rotation, it does not satisfy permutation symmetry as
permuting the entries might change the first element. On the other
hand, though the $1$'s statistic is invariant under permuting all the
entries, it does not exhibit circular rotation invariance. Indeed,
only one out of $n+1$ rotations of an assignment of $n$ cars on a
circle with $n+1$ spots gives a valid parking function. It is thus intriguing what is hidden behind the pair of statistics (leading elements, $1$'s). By casting parking functions in the context of labeled forests, this question will be answered in our paper. We explain the gist of our argument below.

Let $[n]=\{1, 2, \dots,n\}$ and $[n]_0=\{0,1,\dots,n\}$. Let $\calt(n)$ denote the set of all \emph{rooted trees} $T$ on the vertex set $[n]_0$ with root 0. More generally, let $\FC(m, n)$ denote the set of all \emph{rooted forests} $F$ with $n+1$ vertices and $m$ edges (equivalently, $n-m+1$ distinct tree components) such that a specified set of $n-m+1$ vertices are the roots of the different trees. We label the roots of $F$ by $\{01, 02, \dots, 0(n-m+1)\}$ and the non-root vertices by $\{1, 2, \dots, m\}$. The fact that the cardinalities of classical parking functions and of rooted trees are the same, i.e.,
\begin{equation*}
  \vert \PF(n, n) \vert = \vert \calt(n) \vert,
\end{equation*}  
and more generally
\begin{equation*}
  \vert \PF(m, n) \vert = \vert \FC(m, n) \vert,
\end{equation*}  
has motivated much work in the study of connections between the two combinatorial structures. One bijective construction between parking functions and labeled forests goes 
back to Foata and Riordan \cite{Pollak}. Their construction is for the special case $m=n$ and is referred to as a breadth first search (with a queue) on rooted trees. See Yan
\cite[Section 1.2.3]{Yan} and also Chassaing and Marckert \cite{CM}. We will show that under the bijective correspondence induced by breadth first search, the seemingly unrelated leading elements statistic and the $1$'s statistic for classical parking functions both become degree statistics. One of them (the root degree) is classical, while the other (degree of the parent of a fixed vertex) appears to be new.

Generally, statistics based on degrees and other aspects of labeled trees have been studied extensively. A well-known generalization of Cayley's tree theorem includes the degrees of all vertices as additional statistics \cite{HP,Stanley}. Another interesting example is the enumeration of labeled trees with respect to their indegrees: there are two versions of defining indegrees, both leading to the same enumeration formula. In the global orientation (see e.g.~\cite{Stanley}), all edges are oriented towards the root; in the local orientation (see \cite{SZ}), they are oriented towards the higher label. For many more interesting statistics of labeled trees, see for instance \cite{ER} (descents), \cite{MR} (inversions, which were also connected to parking functions in \cite{GL}), or \cite{LP} (runs).

This paper is organized as follows. In Section \ref{sec:pf} we extend the statistic leading elements (denoted $\lel(\pi)$) for classical parking functions that was studied by Stanley and Yin \cite{SY} to general parking functions $\pi \in \PF(m, n)$. We also introduce a new statistic, size of level set (denoted $\slev(\pi)$), for parking functions $\pi \in \PF(m, n)$ that extends the notion of the $1$'s statistic (denoted $\one(\pi)$) for classical parking functions and study the joint distribution of the level set statistic and the leading elements statistic. We establish the generating function for the pair of statistics ($\slev(\pi), \lel(\pi)$) using variations of Pollak's argument. 

In Section~\ref{sec:BFS}, we apply the aforementioned bijection between parking functions and labeled forests that is based on breadth first search. By means of this bijection, we find that the pair of statistics $(\slev(\pi),\lel(\pi))$ on the set of parking functions $\PF(m,n)$ is equidistributed with the pair of statistics $(\deg(0),\deg(p))$ on the set of rooted forests $\FC(m, n)$, where $\deg(0)$ is the root degree of a rooted forest (the total number of children of all roots $01,02,\ldots,0(n-m+1)$), and $\deg(p)$ is the number of children of the parent $p$ of the vertex labeled $1$ (by degree, we generally mean more precisely the number of children of a vertex in a rooted tree, which is $1$ less than the degree in the graph-theoretical sense for non-root vertices).

The pair of statistics $(\deg(0),\deg(p))$ is further considered in Section \ref{sec:forest}. In particular, we directly prove a formula for the number of rooted forests in $\FC(m,n)$ for which $\deg(0)$ and $\deg(p)$ take on given values by means of a combinatorial argument. In the special case $m=n$ (i.e., for labeled trees), the two statistics $\deg(0)$ and $\deg(p)$ also have the same distribution. We provide an explicit bijection for this fact.

In Section \ref{sec:discussions}, we examine the asymptotic properties of the statistics investigated in our paper using standard probabilistic tools.

In Section \ref{sec:rho} we introduce a new bijection on labeled trees, which exchanges the statistics $\deg(0)$ and $\deg(p)$ while preserving the set of children for every other vertex. This bijection leads to a refinement of the symmetric joint distribution of $(\lel(\pi), \one(\pi))$ over classical parking functions.

In the last two sections, we extend our results to $(a, b)$-parking functions. 
The correspondence between parking functions and labeled forests studied earlier in the paper can also be extended to $(a, b)$-parking functions, where the labeled forests are replaced by labeled trees with edge colors. In Section \ref{sec:ab} 
we extend Pollak's proof to a circle with $a+mb$ spots to enumerate $(a,b)$-parking functions of length $m$. This extension gives rise to new and more refined statistics and leads further to the joint distribution of these new statistics. In Section~\ref{sec:rho-ab} we extend the bijection of Section \ref{sec:rho} to explain the symmetries exhibited in those joint distributions.


\section{Statistics on parking functions}
\label{sec:pf}
In this section we investigate the joint distribution of the pair of statistics ($\slev(\pi), \lel(\pi)$) on parking functions $\pi \in \PF(m, n)$. The precise definitions of the individual statistics read as follows:
\begin{itemize}
    \item  Leading elements: total number of cars whose preferred spot is the same as that of the first
car, denoted $\lel(\pi)$. This statistic was recently introduced (in the special case $m=n$) by Stanley and Yin \cite{SY}.
    \item Size of level set: total number of cars whose preferred spot is in the range $\{1, 2, \dots, n-m+1\}$, denoted $\slev(\pi)$. When $m=n$, the level set statistic reduces to the $1$'s statistic for classical parking functions, which counts the total number of cars that prefer spot $1$, often denoted $\one(\pi)$. The level set statistic $\slev(\pi)$ has not been considered before, but the $1$'s statistic $\one(\pi)$ has been widely studied.
\end{itemize}
Our results for $\PF(m, n)$ are extensions of corresponding results for classical parking functions $\PF(n, n)$ in \cite{SY}. As mentioned in the introduction, we will expand upon Pollak's ingenious circle argument \cite{Pollak} for the street parking model to derive our results.

The following lemma was proven before using other means, see for example Kenyon and Yin \cite[Corollary 3.4]{KY}. Our direct combinatorial argument below will shed light on the structure of parking functions
and will be useful in the proof of Theorem~\ref{thm:expl-formula}. As implied by the necessary and sufficient condition~\eqref{pigeon}, changing $\pi_1 = 1$ to any value less than or equal to $n-m+1$ will still keep $\pi$ a parking function. The number of parking functions $\pi \in \PF(m, n)$ with $\pi_1 \in \{1, 2, \dots, n-m+1\}$ is thus $n-m+1$ times the number of parking functions $\pi \in \PF(m, n)$ with
$\pi_1=1$. 

\begin{lemma}\label{simple}
We have
  \begin{equation*}
      \# \{\pi \in \PF(m, n): \pi_1=1\}=(n-m+2)(n+1)^{m-2},   
\end{equation*}
which implies that
\begin{equation*}
      \# \{\pi \in \PF(m, n): \pi_1 \in \{1, 2, \dots, n-m+1\}\}=(n-m+1)(n-m+2)(n+1)^{m-2}.    
\end{equation*}
\end{lemma}

\begin{proof}
The statement is trivially true for $m=1$. For $m\geq 2$, we assign cars $2, \dots, m$ independently on a circle of length
$n+1$. Taking circular rotation into consideration, the car assignments give rise to $(n-m+2)(n+1)^{m-2}$ valid parking functions. Note that car $1$ will always be able to park if its desired spot is spot $1$. Our conclusion readily follows.
\end{proof}

The following lemma allows us to split a parking function in $\PF(m,n)$ into an arbitrary map whose range is precisely the set $\{1,2,\ldots,n-m+1\}$ (that is relevant for the statistic $\slev$) and a parking function on a smaller domain. It will be very useful in proving our results on the distribution of statistics on $\PF(m,n)$.

\begin{lemma}\label{lm:decomp}
Consider a function $\pi: [m] \to [n]$. Fix the elements of $\pi$ that are equal to one of $1, 2, \dots, n-m+1$, and suppose that there are $s\geq 0$ such elements (this is precisely $\slev(\pi)$). Let the other elements be $\pi_{j_1} < \pi_{j_2} < \cdots < \pi_{j_{m-s}}$, and define a new function $\tilde{\pi}: [m-s] \to [m-1]$ by $\tilde{\pi}_i = \pi_{j_i} - (n-m+1)$. Then $\pi$ is a parking function in $\PF(m,n)$ if and only if $\tilde{\pi}$ is a parking function in $\PF(m-s,m-1)$.
\end{lemma}

\begin{remark}
    For $s=0$, there is no valid parking function in view of~\eqref{pigeon}. This is consistent with the fact that $\PF(m,m-1)$ is (trivially) empty.
\end{remark}

\begin{proof}
We make use of the characterization~\eqref{pigeon} of parking functions.
For any $i > m-n$, we have
\begin{equation*}
\#\{k: \pi_k \leq i\} = s + \#\{k: \tilde{\pi}_k \leq i - (n-m+1)\}.      
\end{equation*}
So the condition in~\eqref{pigeon} is equivalent to 
\begin{equation*}
\#\{k: \tilde{\pi}_k \leq i - (n-m+1)\} \geq m-n+i-s
\end{equation*}
for $i=m-n+1,m-n+2,\ldots,n$. Substituting $h = i - (n-m+1)$, this becomes
\begin{equation}\label{eq:modified_cond}
\#\{k: \tilde{\pi}_k \leq h\} \geq h+1-s
\end{equation}
for $h=0,1,\ldots,m-1$. This is precisely the condition for a parking function in $\PF(m-s,m-1)$, except for one detail: the conditions start at $h=0$ rather than $h=s$. However, for $h < s$,~\eqref{eq:modified_cond} is trivially satisfied. This completes the proof.
\end{proof}

Lemma~\ref{lm:decomp} means that every parking function in $\PF(m,n)$ can be uniquely decomposed into an arbitrary function $\pi_a: A \to [n-m+1]$ on a set $A \subseteq [m]$
of cardinality $s$ and a function that is equivalent to a parking function $\pi_p$ in $\PF(m-s,m-1)$.

As a consequence of the decomposition in Lemma~\ref{lm:decomp}, we will now be able to prove results on the distributions of the statistics $\slev(\pi)$ and $\lel(\pi)$. We start with the joint distribution of $\slev(\pi)$ and $\lel(\pi)$, which is determined by the following theorem.

\begin{theorem}\label{thm:expl-formula}
Let $s, t \geq 1$. We have
\begin{multline*} 
\#\{\pi \in \PF(m, n)\colon  \slev(\pi)=s \text{ and } \lel(\pi)=t\}\\=\binom{m-2}{s-1, t-1, m-s-t} (n-m+1)^{s} (m-1)^{m-s-t+1}\\+\binom{m-1}{t-1, s-t, m-s} s (n-m+1) (n-m)^{s-t} m^{m-s-1}. 
\end{multline*}
\end{theorem}

We will revisit this formula later in the context of rooted forests. The generating function for the joint distribution of $\slev(\pi)$ and $\lel(\pi)$ is obtained in a straightforward fashion by summing over all $s$ and $t$.

\begin{corollary}
\label{cor:pf1}
\begin{multline}\label{master:pf1}
\sum_{\pi \in \PF(m, n)} x^{\slev(\pi)} y^{\lel(\pi)}
=(n-m+1)xy\left[(m-1)((n-m+1)x+y+m-1)^{m-2}\right.\\\left.+(xy+(n-m)x+1)(xy+(n-m)x+m)^{m-2}\right].
\end{multline}
\end{corollary}

\begin{proof}[Proof of Theorem~\ref{thm:expl-formula}]
Let us count parking functions $\pi \in \PF(m,n)$ for which $\slev(\pi) = s$. Lemma~\ref{lm:decomp} shows that we can decompose $\pi$ into an arbitrary function $\pi_a$ from $A$ to $[n-m+1]$, where $|A| = s$ and $A \subseteq [m]$, and a (function equivalent to a) parking function $\pi_p$ in $\PF(m-s,m-1)$. For $\lel(\pi)$, we distinguish two cases:
\begin{itemize}
\item The spot of the first car does not lie in $\{1,2,\ldots,n-m+1\}$. In this case, $1 \notin A$, and the value of $\lel(\pi)$ is determined by the function $\pi_p$. Recall that by Pollak's argument, for every possible map from $[m-s]$ to $[m]$, there are $s$ possible rotations that will turn it into a parking function in $\PF(m-s,m-1)$. Thus in a randomly chosen parking function, each car (other than the first) takes the same spot as car $1$ with the same probability $\frac{1}{m}$, and all the cars are independent. So $\lel(\pi)$ follows a binomial distribution in this case, and there are $s \binom{m-s-1}{t-1} (m-1)^{m-s-t}$ possibilities for $\pi_p$ such that $\lel(\pi) = t$. Since there are $\binom{m-1}{s}$ choices for the set $A$ (the domain of $\pi_a$) and $(n-m+1)^s$ choices for the function $\pi_a$ itself, we obtain a total contribution of
\begin{multline*}
\binom{m-1}{s} \cdot (n-m+1)^s \cdot s \binom{m-s-1}{t-1} (m-1)^{m-s-t} \\
= \binom{m-2}{s-1, t-1, m-s-t} (n-m+1)^{s} (m-1)^{m-s-t+1}.    
\end{multline*}
The corresponding generating function is
\begin{multline*}
    \sum_{s=1}^{m} \binom{m-1}{s} (n-m+1)^s y s(m-1+y)^{m-s-1} x^s \\ = (n-m+1)xy (m-1) ((n-m+1)x + y + m - 1)^{m-2}.
\end{multline*}
\item The spot of the first car lies in $\{1,2,\ldots,n-m+1\}$. Then $1 \in A$, and $\lel(\pi)$ is completely determined by the function $\pi_a$. Each element of $A \setminus \{1\}$ has the same probability $\frac{1}{n-m+1}$ of being mapped to the same element as car $1$ by $\pi_a$, and all these elements are independent. So given $s$, $\lel(\pi)$ follows a binomial distribution in this case as well, and given $\pi_a(1)$, there are $\binom{s-1}{t-1}(n-m)^{s-t}$ possibilities for the map $\pi_a$ such that $\lel(\pi) = t$. There are now $\binom{m-1}{s-1}$ choices for the set $A$, $n-m+1$ choices for the spot of the first car, and $sm^{m-s-1}$ possible choices for $\pi_p$, so this case yields a contribution of
\begin{multline*}
    \binom{s-1}{t-1}(n-m)^{s-t} \cdot \binom{m-1}{s-1} (n-m+1) \cdot s m^{m-s-1} \\= \binom{m-1}{t-1, s-t, m-s} s (n-m+1) (n-m)^{s-t} m^{m-s-1}.
\end{multline*}
The generating function associated with this case is
\begin{multline*}
    \sum_{s=1}^{m} \binom{m-1}{s-1} y (n-m+1)(n-m+y)^{s-1}sm^{m-s-1} x^s \\
    = (n-m+1) xy (xy + (n-m)x + m)^{m-2} (xy + (n-m)x + 1).
\end{multline*}
\end{itemize}
Combining the two contributions, we obtain Theorem~\ref{thm:expl-formula} and Corollary~\ref{cor:pf1}.
\end{proof}

Specializing the generating function by setting $x=1$ or $y=1$, we immediately obtain the distributions of $\lel(\pi)$ and $\slev(\pi)$. These are given in the following two corollaries.

\begin{corollary}\label{pf1}
Taking $x=1$ in \eqref{master:pf1}, we have
\begin{equation*}
\sum_{\pi \in \PF(m, n)} y^{\lel(\pi)}=(n-m+1)y(y+n)^{m-1}.
\end{equation*}
\end{corollary}

\begin{corollary}\label{pf2}
Taking $y=1$ in \eqref{master:pf1}, we have
\begin{equation*}
\sum_{\pi \in \PF(m, n)} x^{\slev(\pi)}=(n-m+1)x\left((n-m+1)x+m\right)^{m-1}.
\end{equation*}
\end{corollary}

Observe that the generating functions in Corollaries~\ref{pf1} and~\ref{pf2} are identical for $m = n$ (up to renaming the variable). In this case, $\slev(\pi)$ becomes the statistic $\one(\pi)$ (number of $1$'s in the parking function).

\section{Breadth first search}
\label{sec:BFS}

In this section we explore the implications of the breadth first search (BFS) algorithm connecting parking functions $\PF(m, n)$ and rooted forests $\FC(m, n)$. This allows us to transfer results from parking functions to forests. Our construction will extend the corresponding construction between classical parking functions $\PF(n, n)$ and rooted trees $\calt(n)$ by Foata and Riordan \cite[Section 3]{Pollak}. That our construction is a bijection may be similarly argued as in \cite{Pollak}, with minor adaptations. We will not go over all the technical details here, but will provide the explicit formulas for the generalized construction and illustrate the correspondence with a concrete example.

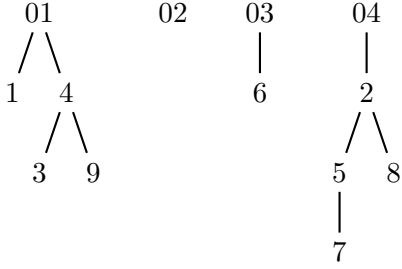
\begin{figure}[htbp]
\begin{center}
\begin{forest}
[01, baseline, [1, edge={thick}], [4, edge={thick} [3, edge={thick}][9, edge={thick}]]]
\end{forest}
\quad
\begin{forest}
[02, baseline]
]
\end{forest}
\quad
\begin{forest}
[03, baseline, [6, edge={thick}]
]
\end{forest}
\quad
\begin{forest}[04, baseline, [2, edge={thick}, [5, edge={thick}, [7, edge={thick}]][8, edge={thick}]]
]
\end{forest}
\end{center}
\caption{\label{illustration1}
Rooted spanning forest.}
\end{figure}

A forest $F \in \FC(m, n)$ may be represented by an acyclic function $f$, where for a non-root vertex $i$, $f_i=j$ indicates that vertex $j$ is the parent of vertex $i$ in a tree component of the forest. Take $m=9$ and $n=12$. See Figure \ref{illustration1} representing an element of $F \in \FC(9,12)$, which corresponds to the acyclic function $f$ given below:
\begin{equation}
\begin{tabular}{ccccccccccc}
$i$ & $=$ & 1 & 2 & 3 & 4 & 5 & 6 & 7 & 8 & 9 \\
$f_i$ & $=$ & 01 & 04 & 4 & 01 & 2 & 03 & 5 & 2 & 4
\end{tabular}.
\label{acyclic}
\end{equation}
We read the vertices of the forest in breadth first search (BFS) order. That is, read root vertices in order first, then all vertices at level $1$ (children of a root), then those at level $2$ (distance $2$ from a root), and so on, where vertices at a given level are naturally ordered in order of increasing predecessor, and, if they have the same predecessor, increasing order. See \cite[Section 1.2.3]{Yan} for a description of this graph searching algorithm in the language of computer science. Applying BFS to the forest $F$ in Figure \ref{illustration1}, we have
\begin{equation*}
v_{01},\dots, v_{04}, v_5, \dots, v_{13} = 01,02,03,04,1,4,6,2,3,9,5,8,7.
\end{equation*}
We let $\sigma^{-1}_f$ be the vertex ordering once we remove the root vertices and $\sigma_f$ be the \emph{inverse order permutation} of $\sigma^{-1}_f$.
\begin{equation*}
\begin{tabular}{ccccccccccc}
$i$ & $=$ & 1 & 2 & 3 & 4 & 5 & 6 & 7 & 8 & 9 \\
$\sigma^{-1}_f(i)$ & $=$ & 1 & 4 & 6 & 2 & 3 & 9 & 5 & 8 & 7 \\
$\sigma_f(i)$ & $=$ & 1 & 4 & 5 & 2 & 7 & 3 & 9 & 8 & 6
\end{tabular}.
\end{equation*}
We further let $t(f)=(r_1, \dots, r_{12})$ with $r_i$ recording the degree of $v_i$, starting with $v_{01}$ and ending with $v_{12}$ (ignoring the final vertex $v_{13}$), that is,
\[t(f)=(2, 0, 1, 1, 0, 2, 0, 2, 0, 0, 1, 0).\]
The sequence $t(f)$ is referred to as the \emph{forest specification} of $F$.

Via the breadth first search, a generic forest $F \in \FC(m, n)$ may thus be uniquely characterized by its associated specification $t(f)$ and order permutation $\sigma_f$. Furthermore, the pair $(t(f), \sigma_f)$ must satisfy certain \emph{balance} and \emph{compatibility} conditions. For exact definitions of these conditions, see~\cite[Section 2.2]{KY} and the references therein. Indeed, if we let $\mathscr{C}(m, n)$ be the set of all feasible pairs, then $\mathscr{C}(m, n)$ is in one-to-one correspondence with $\FC(m, n)$. It turns out that $\mathscr{C}(m, n)$ is also in one-to-one correspondence with the set of parking functions $\PF(m, n)$, which we now describe.

{For a parking function $\pi \in \PF(m, n)$, the associated \emph{specification} is $s(\pi)=(\#1(\pi), \dots, \#n(\pi))$, where $\#k(\pi)=\#\{i: \pi_i=k\}$ records the number of cars whose parking preference is spot $k$. The \emph{order permutation} $\tau_\pi \in \Sym_m$, on the other hand, is defined by
$\tau_\pi(i)=\# \{j: \pi_j<\pi_i, \text{ or } \pi_j=\pi_i \text{ and } j\leq i\}$,
and so is the permutation that orders the list, without switching elements that are the same. In words, $\tau_\pi(i)$ is the position of the entry $\pi_i$ in the non-decreasing rearrangement of $\pi$. For example, for $\pi=(3, 1, 3, 1)$, $\tau_\pi(1)=3$, $\tau_\pi(2)=1$, $\tau_\pi(3)=4$, and $\tau_\pi(4)=2$. We can easily recover a parking function $\pi$ by replacing $i$ in $\tau_\pi$ with the $i$-th smallest term in the sequence $1^{\#1}\dots n^{\#n}$. As in the case of rooted forests, all feasible pairs $(s(\pi), \tau_\pi)$ for parking functions constitute the set $\mathscr{C}(m, n)$.}

Combining the above perspectives, we see that the breadth first search algorithm bijectively connects parking functions and rooted forests, where $(t(f), \sigma_f)=(s(\pi), \tau_\pi)$. Continuing with our earlier example, for the forest $F \in \FC(9, 12)$ in Figure \ref{illustration1} with acyclic function representation $f$ given by \eqref{acyclic}, we have
\[s(\pi)=(2, 0, 1, 1, 0, 2, 0, 2, 0, 0, 1, 0),\]
and
\begin{equation*}
\begin{tabular}{ccccccccccc}
$i$ & $=$ & 1 & 2 & 3 & 4 & 5 & 6 & 7 & 8 & 9 \\
$\tau^{-1}_\pi(i)$ & $=$ & 1 & 4 & 6 & 2 & 3 & 9 & 5 & 8 & 7 \\
$\tau_\pi(i)$ & $=$ & 1 & 4 & 5 & 2 & 7 & 3 & 9 & 8 & 6
\end{tabular}.
\end{equation*}
We form the non-decreasing rearrangement sequence associated with $s(\pi)$: 
\begin{equation*}
1^2, 3^1, 4^1, 6^2, 8^2, 11^1=1, 1, 3, 4, 6, 6, 8, 8, 11.
\end{equation*}
Replacing $i$ in $\tau_\pi$ with the $i$-th smallest term in this sequence yields the corresponding parking function $\pi \in \PF(9, 12)$ given below:
\begin{equation*}
\begin{tabular}{ccccccccccc}
$i$ & $=$ & 1 & 2 & 3 & 4 & 5 & 6 & 7 & 8 & 9 \\
$\pi_i$ & $=$ & 1 & 4 & 6 & 1 & 8 & 3 & 11 & 8 & 6
\end{tabular}.
\end{equation*}

{This bijective construction (let us name it $\phi$) from rooted forests to parking functions has some interesting implications that are listed in the following theorem.}

\begin{theorem}\label{thm:correspondence}
The following statistics are equidistributed:
\begin{itemize}
\item The number of times $\pi_i$ appears in a parking function $\pi \in \PF(m, n)$ equals the degree of the parent of vertex $i$ in the corresponding forest $F \in \FC(m, n)$.
\item The number of times $1, 2, \dots, n-m+1$ appears in a parking function $\pi \in \PF(m, n)$ respectively equals the degree of the root vertex $01, 02, \dots, 0(n-m+1)$ in the corresponding forest $F \in \FC(m, n)$.
\end{itemize}
\end{theorem}

\begin{proof}
This is due to our specific construction. From a forest $F$ to a parking function $\pi$, we have
\begin{equation*}
\pi_i=\left\{\begin{array}{ll}
 j    & \text{ if } f_i=0j \text{ for some }j=1, 2, \dots, n-m+1, \\
 (n-m+1)+\sigma_f(f_i)    & \text{ otherwise}.
\end{array}\right.
\end{equation*}
Conversely, from a parking function $\pi$ to a forest $F$, we have
\begin{equation*}
f_i=\left\{\begin{array}{ll}
 0j    & \text{ if } \pi_i=j \text{ for some }j=1, 2, \dots, n-m+1, \\
 \tau_\pi^{-1}(\pi_i-(n-m+1))    & \text{ otherwise}.
\end{array}\right.
\end{equation*}
The second claim is clear. For the first claim, we note that $\pi_i=\pi_j$ corresponds to $f_i=f_j$, i.e., vertices $i$ and $j$ have the same parent.
\end{proof}

\begin{remark}
Note that in our construction, $\pi_i=j$, or car $i$ prefers spot $j$, if and only if $i$ is a child of $u$, where $u$ is the $j$-th vertex of the rooted forest under the BFS order. Theorem \ref{thm:correspondence} thus follows immediately.
\end{remark}


\section{Statistics on trees and forests}

In the bijection described in the previous section, the number of times $\pi_1$ occurs in the parking function (the statistic $\lel(\pi)$) corresponds to the degree $\deg(p)$ of the parent $p$ of vertex $1$ (see Theorem~\ref{thm:correspondence}). Moreover, the total number of times $1,2,\ldots,n-m+1$ occur in the parking function (our statistic $\slev(\pi)$) corresponds to the total root degree $\deg(0)$, i.e., the sum of the degrees of all roots. Hence, the pair $(\deg(0),\deg(p))$ follows the same joint distribution as the pair $(\slev(\pi),\lel(\pi))$. The following generating function identity is therefore an automatic consequence of Corollary~\ref{cor:pf1}.

\label{sec:forest}

\begin{theorem}\label{thm:all3}
\begin{multline}\label{master3}
\sum_{F \in \FC(m, n)} x^{\deg(0)} y^{\deg(p)}
=(n-m+1)xy\left[(m-1)((n-m+1)x+y+m-1)^{m-2}\right.\\\left.+(xy+(n-m)x+1)(xy+(n-m)x+m)^{m-2}\right].
\end{multline}
\end{theorem}

When $m=n$, the rooted spanning forest $F \in \FC(m, n)$ reduces to a rooted tree $T \in \calt(n)$. We recognize from Corollaries \ref{pf1} and \ref{pf2} that $\one(\pi)$ and $\lel(\pi)$ are equidistributed. The breadth first search algorithm maps them to $\deg(0)$ and $\deg(p)$, respectively, so those are equidistributed as well. In words, the distribution of the number of children of the root $0$ of a rooted labeled tree follows the same distribution as the number of children of the parent of vertex $1$ (or indeed by symmetry the parent of any fixed vertex). The property of being parent of a specific vertex induces a bias towards higher degrees, which turns out to be equivalent to the bias induced by being the root (which necessarily has at least one child). The following procedure provides an explicit bijection $\theta$ for the equidistribution of $\deg(0)$ and $\deg(p)$.
\begin{enumerate}
    \item Remove the edge connecting vertices $1$ and $p$;
    \item Connect vertices $0$ and $1$ by an edge;
    \item Interchange vertices $0$ and $p$. 
\end{enumerate}

\begin{figure}[htbp]
\centering
\begin{tikzpicture}[scale=0.8]
\node at (0,0) {$0$} [grow = down]
    child {edge from parent [thick]}
    child {edge from parent [thick]}
    child {node {$p$} edge from parent [dashed, thick]
    child{edge from parent [solid, thick]}
    child{edge from parent [solid, thick]}
    child{node {$1$} edge from parent [color=red, solid, thick]}};
    \draw[thick] (2.7, -3.3) to (1.5, -4.5);
    \draw[thick] (3, -3.3) to (3, -4.5);
    \draw[thick] (3.3, -3.3) to (4.5, -4.5);

\draw [<->, thick] (4.75,-2) -- (6.75,-2);
\node at (5.75, -1.5) {$\theta$}; 

\node at (9,0) {$p$} [grow = down]
    child {edge from parent [thick]}
    child {edge from parent [thick]}
    child {node {$0$} edge from parent [dashed, thick] child{edge from parent
        [solid, thick]} child{edge from parent [solid, thick]} child{node {$1$} child{edge from parent [thick]} child{edge from parent [thick]} child{edge from parent [thick]} edge from parent [white, densely dotted, thick]}};
    \draw[color=red, thick] (9.2,0.15) to[bend left=90] (12.15,-2.9);
    \draw[thick] (11.7, -3.3) to (10.5, -4.5);
    \draw[thick] (12, -3.3) to (12, -4.5);
    \draw[thick] (12.3, -3.3) to (13.5, -4.5);
\end{tikzpicture}
\caption{A bijective map between $\deg(0)$ and $\deg(p)$ (illustration).}
\label{tree:illustration}
\end{figure}
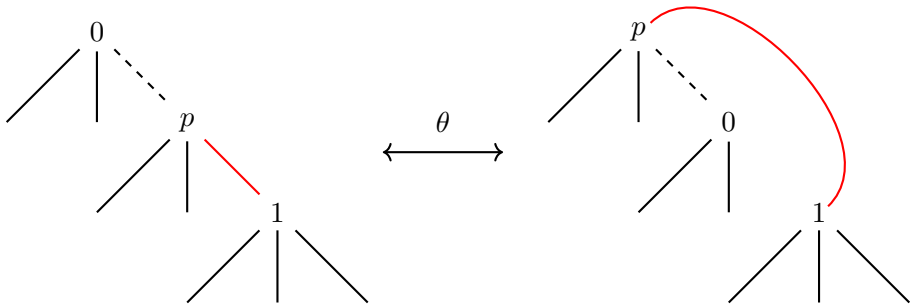

See Figure~\ref{tree:illustration} for the general procedure and Figure~\ref{tree:example} for an example. This map has the extra benefit of being an involution. Moreover, the degrees of all vertices except $0$ and $p$ are preserved. Some nice features are hence introduced in the corresponding parking function bijection, where in our example
\begin{equation*}
    \pi=(8, 4, 5, 1, 2, 1, 1, 5, 6) \leftrightarrow \pi'=(5, 8, 2, 2, 1, 5, 5, 4, 9).
\end{equation*}
We see that $\one(\pi)$ and $\lel(\pi)$ are switched, but the frequencies of the non-$1$ and non-leading elements are preserved up to permutation. In the example, the non-$1$ and non-leading elements in $\pi$ are $5$ (occurring twice), $2$, $4$, and $6$. Those in $\pi'$ are $2$ (occurring twice), $4$, $8$, and $9$.

\begin{figure}[htbp]
\centering
\begin{tikzpicture}[scale=0.6]
\node at (0,0) {0} [grow = down]
    child {node{4} edge from parent [thick] child{node{5} edge from parent [thick] child {node{3} edge from parent [thick]} child {node{8} edge from parent [thick] child {node{1} edge from parent [thick]}}}}
    child {node{6} edge from parent [thick]}
    child {node{7} edge from parent [thick] child{node{2} edge from parent [thick] child{node{9} edge from parent [thick]}}};

\draw [<->, thick] (3,-2) -- (5,-2);
\node at (4, -1.5) {$\theta$}; 

\node at (7.5,1.5) {0} [grow = down]
    child {node{5} edge from parent [thick] child {node{3} edge from parent [thick]} child{node{4} edge from parent [thick] child {node{8} edge from parent [thick] child {node{1} edge from parent [thick]} child {node{6} edge from parent [thick]} child {node{7} edge from parent [thick] child {node{2} edge from parent [thick] child {node{9} edge from parent [thick]}}}}}};
\end{tikzpicture}
\caption{An example of the map $\theta$.}
\label{tree:example}
\end{figure}
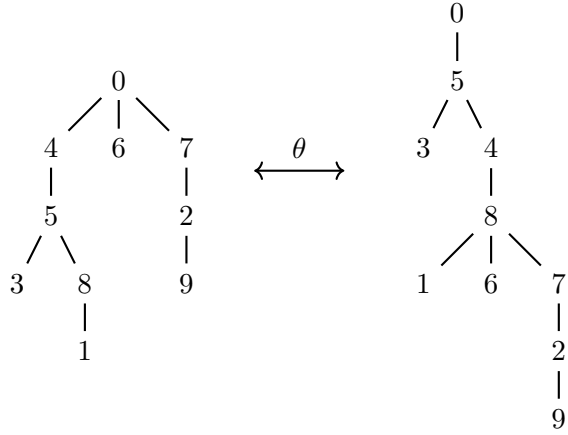

The counting formula in our next proposition is equivalent to Theorem~\ref{thm:expl-formula} in view of the bijection between parking functions and forests. In the following, we also illustrate it with a combinatorial proof of the statement in the setting of forests.

\begin{proposition}\label{thm:master3alt}
Let $s, t \geq 1$. We have
\begin{multline*} 
\#\{F \in \FC(m, n)\colon  \deg(0)=s \text{ and } \deg(p)=t\}\\=\binom{m-2}{s-1, t-1, m-s-t} (n-m+1)^{s} (m-1)^{m-s-t+1}\\+\binom{m-1}{t-1, s-t, m-s} s (n-m+1) (n-m)^{s-t} m^{m-s-1}. 
\end{multline*}
\end{proposition}

\begin{proof}
As a starting point, we recall the well-known result that the number of rooted forests with vertex set $[a]$ and $b$ components whose root labels are given is $ba^{a-b-1}$ (see~\cite{Stanley}). Thus the number of such rooted forests \emph{with one distinguished vertex} (possibly one of the roots) is $ba^{a-b}$, and by symmetry the number of such rooted forests where a vertex in the first component is distinguished must be $a^{a-b}$. 
Now in order to prove our formula, we distinguish two cases:

\textbf{Case 1. The parent $p$ of vertex $1$ is not one of the roots}. 
A forest $F$ with this property as well as $\deg(0)=s$  and  $\deg(p)=t$ can be uniquely constructed as follows:
\begin{itemize}
\item Choose a label $r \in [m] \setminus \{1\}$ ($m-1$ possibilities).
\item Choose two disjoint sets of labels $\{x_1,x_2,\ldots,x_{s-1}\}$ and $\{y_1,y_2,\ldots,y_{t-1}\}$ from $[m] \setminus \{1,r\}$ (for which there are $\binom{m-2}{s-1,t-1,m-s-t}$ possibilities).
\item Choose a rooted forest on $[m] \setminus \{1\}$ with root labels $r,x_1,x_2,\ldots,x_{s-1},y_1,y_2,\ldots,y_{t-1}$ and a distinguished vertex $p$ in the first component (there are $(m-1)^{m-s-t}$ possibilities, as explained above). Note that potentially $p = r$.
\item Split the distinguished vertex $p$ into two vertices, labeled $p$ and $1$ respectively, where $p$ is the parent and $1$ is the child, and all former children of $p$ now become children of $1$.
\item Add roots $01, 02, \dots, 0(n-m+1)$, and connect each of the vertices $r,x_1,x_2,\ldots,x_{s-1}$ with one of these roots by an edge. There are $(n-m+1)^s$ possibilities for this.
\item Add edges between vertex $p$ and $y_1,y_2,\ldots,y_{t-1}$. Note that $p$ is indeed the parent of $1$ in this construction, and that $\deg(0)=s$ as well as $\deg(p)=t$.
\end{itemize}
It is easy to reverse the procedure given a forest with $\deg(0) = s$ and $\deg(p) = t$ for which $p$ is not one of the roots. 
So to summarize, we have
\begin{equation*}
    \binom{m-2}{s-1, t-1, m-s-t} (n-m+1)^{s} (m-1)^{m-s-t+1}
\end{equation*}
possible forests in this case, which accounts for the first term in our formula. Case 1 is illustrated in the special case $n = m = 13$, $s = 3$ and $t = 2$ in Figure~\ref{fig:illustration}. The choice of root labels is $r = 5$, $\{x_1,x_2\} = \{3,13\}$ and $\{y_1\} = \{9\}$. 

\begin{figure}[htbp]
\begin{center}
\begin{tikzpicture}[scale=0.6]

\node [circle,draw,inner sep=1.5mm] (v1) at (0,0) {5};
\node [circle,draw,inner sep=1.5mm] (v2) at (4.5,0) {3};
\node [circle,draw,inner sep=1mm] (v3) at (9,0) {13};
\node [circle,draw,inner sep=1.5mm] (v4) at (13.5,0) {9};
\node [circle,draw,inner sep=1mm] (v5) at (-1.5,-2) {12};
\node [circle,draw,ultra thick,inner sep=1.5mm] (v6) at (1.5,-2) {7};
\node [circle,draw,inner sep=1.5mm] (v7) at (4.5,-2) {8};
\node [circle,draw,inner sep=1.5mm] (v8) at (7.5,-2) {2};
\node [circle,draw,inner sep=1.5mm] (v9) at (10.5,-2) {6};
\node [circle,draw,inner sep=1mm] (v10) at (13.5,-2) {11};
\node [circle,draw,inner sep=1.5mm] (v11) at (1.5,-4) {4};
\node [circle,draw,inner sep=1mm] (v12) at (4.5,-4) {10};

\draw (v1)--(v5);
\draw (v1)--(v6);
\draw (v6)--(v11);
\draw (v2)--(v7);
\draw (v7)--(v12);
\draw (v3)--(v8);
\draw (v3)--(v9);
\draw (v4)--(v10);

\node at (0,0.9) {$r$};
\node at (4.5,0.9) {$x_1$};
\node at (9,0.9) {$x_2$};
\node at (13.5,0.9) {$y_1$};
\node at (1.6,-1.1) {$p$};

\draw [dashed] (3.75,1.5)--(5.25,1.5)--(5.25,-5)--(3.75,-5)--cycle;
\draw [dashed] (6.75,1.5)--(11.25,1.5)--(11.25,-3)--(6.75,-3)--cycle;
\draw [dashed] (12.75,1.5)--(14.25,1.5)--(14.25,-3)--(12.75,-3)--cycle;

\node [circle,draw,inner sep=1.5mm] (w0) at (4.5,-6) {0};
\node [circle,draw,inner sep=1.5mm] (w1) at (0,-8) {5};
\node [circle,draw,inner sep=1.5mm] (w2) at (4.5,-8) {3};
\node [circle,draw,inner sep=1mm] (w3) at (9,-8) {13};
\node [circle,draw,inner sep=1.5mm] (w4) at (2.25,-12) {9};
\node [circle,draw,inner sep=1mm] (w5) at (-1.5,-10) {12};
\node [circle,draw,ultra thick,inner sep=1.5mm] (w6) at (1.5,-10) {7};
\node [circle,draw,inner sep=1.5mm] (w7) at (4.5,-10) {8};
\node [circle,draw,inner sep=1.5mm] (w8) at (7.5,-10) {2};
\node [circle,draw,inner sep=1.5mm] (w9) at (10.5,-10) {6};
\node [circle,draw,inner sep=1mm] (w10) at (2.25,-14) {11};
\node [circle,draw,inner sep=1.5mm] (w11) at (0.75,-12) {1};
\node [circle,draw,inner sep=1mm] (w12) at (4.5,-12) {10};
\node [circle,draw,inner sep=1.5mm] (w13) at (0.75,-14) {4};

\draw (w0)--(w1);
\draw (w0)--(w2);
\draw (w0)--(w3);
\draw (w1)--(w5);
\draw (w1)--(w6);
\draw (w6)--(w11);
\draw (w2)--(w7);
\draw (w7)--(w12);
\draw (w3)--(w8);
\draw (w3)--(w9);
\draw (w6)--(w4);
\draw (w4)--(w10);
\draw (w11)--(w13);

\draw [dashed] (1.5,-11)--(3,-11)--(3,-15)--(1.5,-15)--cycle;
\draw [dashed] (3.75,-7)--(5.25,-7)--(5.25,-13)--(3.75,-13)--cycle;
\draw [dashed] (6.75,-7)--(11.25,-7)--(11.25,-11)--(6.75,-11)--cycle;
\end{tikzpicture}
\end{center}
\caption{Illustration of the procedure: the rooted forest (top) with the distinguished vertex $p$ indicated by a thick node, and the final tree (bottom).}\label{fig:illustration}
\end{figure}
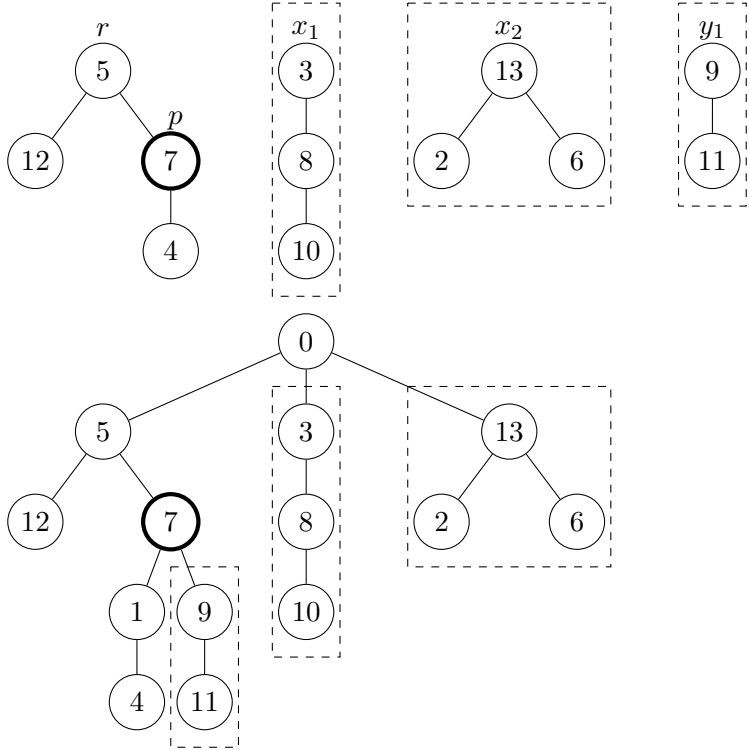

\textbf{Case 2. The parent $p$ of vertex $1$ is one of the roots}. 
Again, there is a unique way to construct all these forests: 
\begin{itemize}
    \item Select a set of labels $\{x_1,x_2,\ldots,x_{s-1}\}$ from $[m] \setminus \{1\}$ in $\binom{m-1}{s-1}$ ways.
    \item Construct a rooted forest with vertex set $[m]$ and root labels $1,x_1,x_2,\ldots,x_{s-1}$ (there are $sm^{m-s-1}$ possibilities).
    \item Among the labels $x_1,x_2,\ldots,x_{s-1}$, choose the siblings of vertex $1$; there are $\binom{s-1}{t-1}$ possible choices.
    \item Pick one of the $n-m+1$ roots $01,02,\ldots,0(n-m+1)$ as the parent $p$ of vertex $1$, and connect it and all the siblings chosen in the previous step to it by an edge.
    \item Lastly, pick one of the other $n-m$ roots as parent for each of the remaining vertices with label in the set $\{x_1,x_2,\ldots,x_{s-1}\}$. This step yields $(n-m)^{s-t}$ possibilities.
\end{itemize}
Putting these together gives us the second term and thus completes the proof.
\end{proof}

\section{Limit distributions}
\label{sec:discussions}

In this section we make some asymptotic observations on the statistics investigated in
our paper. We consider the scenario where $m$ is a linear function of $n$, and determine the limit distributions of the statistics $\lel$ and $\slev$ (for parking functions) and $\deg(0)$ as well as $\deg(p)$ (for rooted forests), respectively.

\begin{proposition}\label{Poisson-CLT-1}
Let $s$ be a fixed positive integer, and take $m=cn$ for some $0<c \leq 1$ as $n \to \infty$.
Consider the parking preference $\pi \in \PF(m, n)$ chosen uniformly at random, and let $\lel_s(\pi):=\#s(\pi)$ be the number of cars with the same preference as car
$s$. Then $\lel_s(\pi) - 1 \dconv \mathrm{Poisson}(c)$, i.e., for every fixed nonnegative integer $j$, 
\begin{equation*}
\PR\left(\lel_s(\pi)=1+j \hspace{.1cm} \vert \hspace{.1cm} \pi \in \PF(m, n) \right) \sim \frac{c^j e^{-c}}{j!}.
\end{equation*}
\end{proposition}

\begin{proof}
By permutation symmetry, we only need to prove this result for $s=1$, where $\lel_s(\pi)$ is exactly $\lel(\pi)$. We divide both sides of the generating function for $\lel(\pi)$ in Corollary \ref{pf1} through by $(n-m+1)(n+1)^{m-1}$. The right side becomes the probability generating function of $S(m, n)=1+\sum_{i=1}^{m-1} X_i$, where the $X_i$ are independent Bernoulli random variables:
\begin{equation*}
X_i=\left\{
      \begin{array}{ll}
        0, & \hbox{probability $n/(n+1)$,} \\
        1, & \hbox{probability $1/(n+1)$.}
      \end{array}
    \right.
\end{equation*}
Hence we have a standard case of the law of rare events, leading to a Poisson limit distribution.
\end{proof}

\begin{corollary}\label{Poisson-CLT-3}
Let $s$ be a fixed positive integer, and take $m=cn$ for some $0<c \leq 1$ as $n \to \infty$.
Consider the labeled forest $F \in \FC(m, n)$ chosen uniformly at random, and let $\deg(p_s)$ be the degree of the parent $p_s$ of vertex $s$. Then $\deg(p_s) - 1 \dconv \mathrm{Poisson}(c)$.
\end{corollary}

\begin{proof}
By permutation symmetry, we only need to prove this result for $s=1$, where $\deg(p_s)$ is exactly $\deg(p)$.
Since $\lel(\pi)$ and $\deg(p)$ are equidistributed (by Theorem~\ref{thm:correspondence}), the statement follows from Proposition \ref{Poisson-CLT-1}.
\end{proof}

\begin{proposition}\label{Poisson-CLT-2}
Take $m=cn$ for some $0<c<1$ as $n \to \infty$.
Consider the parking preference $\pi \in \PF(m, n)$ chosen uniformly at random. Then we have 
\begin{equation*}
    \frac{\slev(\pi)-c(1-c)n}{\sqrt{c^2(1-c)n}} \dconv \mathcal{N}(0,1).
\end{equation*}
\end{proposition}

\begin{proof}
We proceed as in the proof of Proposition \ref{Poisson-CLT-1} and divide both sides of the generating function for $\slev(\pi)$ in Corollary \ref{pf2} through by $(n-m+1)(n+1)^{m-1}$. The right side becomes the probability generating function of $S(m, n)=1+\sum_{i=1}^{m-1} X_i$, where the $X_i$ are independent Bernoulli random variables:
\begin{equation*}
X_i=\left\{
      \begin{array}{ll}
        0, & \hbox{probability $m/(n+1)$,} \\
        1, & \hbox{probability $(n-m+1)/(n+1)$.}
      \end{array}
    \right.
\end{equation*}
The probabilities converge to $c$ and $1-c$ respectively, and the standard central limit theorem applies. This means that $\slev(\pi)$ may be approximated by $\mathcal{N}(0, 1)$ after standardization.
\end{proof}

\begin{corollary}\label{Poisson-CLT-4}
Take $m=cn$ for some $0<c<1$ as $n \to \infty$.
Consider the labeled forest $F \in \FC(m, n)$ chosen uniformly at random. Then we have
\begin{equation*}
    \frac{\deg(0)-c(1-c)n}{\sqrt{c^2(1-c)n}} \dconv \mathcal{N}(0,1).
\end{equation*}
\end{corollary}

\begin{proof}
Since $\slev(\pi)$ and $\deg(0)$ are equidistributed (again by Theorem~\ref{thm:correspondence}), this is the same proof as for Proposition \ref{Poisson-CLT-2}.
\end{proof}

\begin{remark}
The asymptotic analysis of the special case $\one(\pi)$ for $\pi \in \PF(n, n)$ was conducted by Diaconis and Hicks \cite{DH}. The limit distribution of the root degree in labeled trees is classical \cite[Example IX.6]{FS}. Both $\one(\pi)-1$ and $\deg(0)-1$ can be approximated by $\mathrm{Poisson}(1)$.
\end{remark}

\section{Another bijection on labeled trees} \label{sec:rho} 

In this section we describe another simple involution $\rho$ on $\calt(n)$, which exchanges $\deg(0)$ and $\deg(p)$, where $p$ is the parent of vertex 1, while preserving the set of children, and hence the degree,  of every other vertex. 

For a classical parking function $\pi$, let $K(\pi)=\{ j: \pi_j=1 \text{ or } \pi_1\}$ be the set of cars that prefers spot $1$ or $\pi_1$. 
Translating the map $\rho$  to classical parking functions via the BFS bijection described in Section \ref{sec:BFS}, we obtain a
new involution $\hat \rho: \pi \rightarrow \hat \pi$ with the properties: 
\begin{itemize}
    \item  The map 
    $\hat \rho$ preserves the set $K(\pi)$, i.e., 
    $K(\pi)= K(\hat \pi)$.
    
    \item For any pair $k, \ell  \not \in K(\pi)$, $\pi_k = \pi_\ell$ if and only if  $\hat \pi_k=\hat \pi_\ell $.  
\end{itemize}

 Given  a rooted tree $T \in \calt(n)$, let the path from vertex $0$ to vertex $1$ be $0 - q - \cdots -p - 1$.  Assume that other children of vertex $0$ are $t_1, \dots, t_c$, and other children of vertex $p$ are $s_1, \dots, s_d$. 
The bijection $\rho$ is simply defined as follows: While $p \neq 0$, 
\begin{enumerate}
    \item replace the edges $(0, t_i)$ with $(p, t_i)$; 
    \item replace the edges  $(p, s_i)$ with $(0, s_i)$; 
    \item all other edges are kept the same. 
\end{enumerate}
In other words, the map $\rho$ swaps  the subtrees rooted at $t_i$ with those rooted as $s_j$, while keeping other branches of the rooted tree unchanged. 

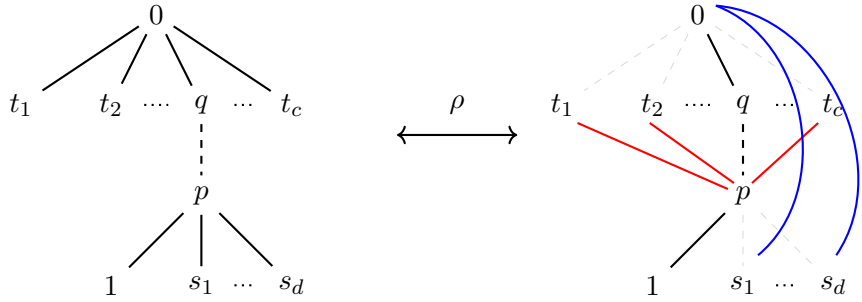
\begin{figure}[htbp]
\centering
\begin{tikzpicture}[scale=0.8]
\node at (0,0) {$0$} [grow = down]
   child {node {$t_1$} edge from parent [thick]}
   child {node {$t_2$} edge from parent [thick]}
    child { node {$q$} edge from parent [thick]
       child{node {$p$} edge from parent [dashed, thick]
           child{node {$1$} edge from parent [solid, thick]}
           child{node {$s_1$} edge from parent [solid, thick]}
           child{node {$s_d$} edge from parent [solid, thick]}}}  
   child {node {$t_c$} edge from parent [thick]};
    
  \draw [dotted, thick] (1.3, -4.5)--(1.6,-4.5); 
  \draw [dotted, thick] (1.3, -1.5)--(1.6, -1.5) (-0.2,-1.5)--(0.2,-1.5); 
  
\draw [<->, thick] (4,-2) -- (6,-2);
\node at (5, -1.5) {$\rho$}; 

\node at (9,0) {$0$} [grow = down]
child {node {$t_1$} edge from parent [color=gray!30, dashed, very thin]}
   child {node {$t_2$} edge from parent [color=gray!30, dashed, very thin]}
    child { node {$q$} edge from parent [thick]
       child{node {$p$} edge from parent [dashed, thick]
           child{node {$1$} edge from parent [solid, thick]}
           child{node {$s_1$} edge from parent [color=gray!30, dashed, very thin]}
           child{node {$s_d$} edge from parent [color=gray!30, dashed, very thin]}}}  
   child {node {$t_c$} edge from parent [color=gray!30, dashed, very thin]};
   \draw [dotted, thick] (10.3, -4.5)--(10.6,-4.5); 
  \draw [dotted, thick] (8.8,-1.5)--(9.2,-1.5) (10.3, -1.5)--(10.6, -1.5);
  
  \draw [color=red, thick] (9.5, -2.9)--(7, -1.8); 
  \draw [color=red, thick] (9.6, -2.8)--(8.2, -1.8);
  \draw [color=red, thick] (9.9, -2.8)--(11, -1.8);
  
  \draw[color=blue,  thick] (9.3,0.15) to[bend left=60] (10,-4);
  \draw[color=blue,  thick] (9.3,0.15) to[bend left=60] (11.3,-4);
  
\end{tikzpicture}
\caption{The involution $\rho$ swaps $\deg(0)$ and $\deg(p)$ (illustration).}
\label{rho:illustration}
\end{figure}

For any vertex $q \neq 0$ or $p$, let $C(q)$ be the set of children of $q$. Then  the nonempty sets of the form $C(q)$, together with $K=C(0) \cup C(p)$, form a set partition of $[n]$, which is preserved under $\rho$.  In terms of parking functions, for $\pi \in \PF(n)$, let $B_\pi(i)=\{j: \pi_j=i\}$ be the set of cars that prefer the spot $i$.  The collection of non-empty $B_\pi(i)$'s 
are exactly the collection of non-empty $C(q)$'s, 
 which we call the \emph{preference partition} of $\pi$. 
Now merge the blocks $B_\pi(1)$ and $B_\pi(\pi_1)$ to obtain the \emph{reduced preference partition} of $\pi$. 
(If $\pi_1=1$, the reduced preference partition is the same as the preference partition.) 
Then the bijection $\hat \rho$ preserves the reduced preference partition. 
Hence we have the following theorem which is a refinement of the symmetric joint distribution of $\lel(\pi)$ and $\mathrm{ones}(\pi)$. 


\begin{theorem}  \label{thm:map-rho} 
 Fix a set partition $\mathcal{P}=\{P_1, P_2, \dots, P_t\}$ of $[n]$. 
Then $\lel(\pi)$ and $\mathrm{ones}(\pi)$  have a symmetric joint distribution over all classical parking functions whose reduced preference partition is $\mathcal{P}$. 


\end{theorem}

Figure ~\ref{fig:map-rho} shows the image of $\rho$ for the rooted tree on the left of Figure~\ref{tree:example}. 
Note that it is different from the image of $\theta$ that is displayed on the right of Figure~\ref{tree:example}.

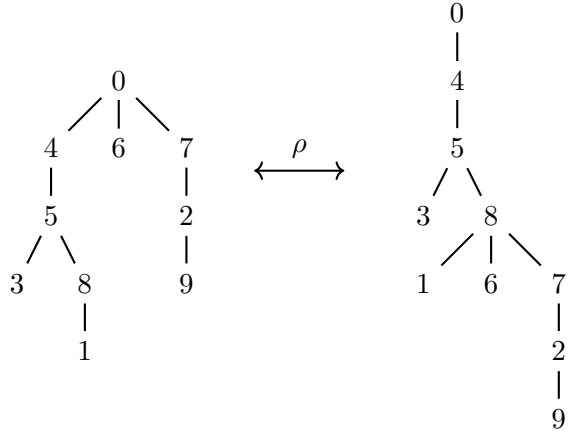
\begin{figure}[htbp]
\centering
\begin{tikzpicture}[scale=0.6]
\node at (0,0) {0} [grow = down]
    child {node{4} edge from parent [thick] child{node{5} edge from parent [thick] child {node{3} edge from parent [thick]} child {node{8} edge from parent [thick] child {node{1} edge from parent [thick]}}}}
    child {node{6} edge from parent [thick]}
    child {node{7} edge from parent [thick] child{node{2} edge from parent [thick] child{node{9} edge from parent [thick]}}};

\draw [<->, thick] (3,-2) -- (5,-2);
\node at (4, -1.5) {$\rho$}; 

\node at (7.5,1.5) {0} [grow = down]
    child { node{4} edge from parent [thick] 
            child {node{5} edge from parent [thick] 
            child {node{3} edge from parent [thick] }
            child {node{8} edge from parent [thick] 
                 child {node{1} edge from parent [thick]}
                 child {node{6} edge from parent [thick]} 
                 child {node{7} edge from parent [thick] 
                      child {node{2} edge from parent [thick] 
                      child {node{9} edge from parent [thick]} } 
                       }
                  }
                }
            };
                      
\end{tikzpicture}
\caption{An example of the map $\rho$.}
\label{fig:map-rho}
\end{figure}

Under the BFS bijection, the above trees give 
\begin{equation*}
    \pi=(8, 4, 5, 1, 2, 1, 1, 5, 6) \leftrightarrow \hat \rho(\pi)=\hat \pi=(5, 8, 3, 1, 2, 5, 5, 3, 9).
\end{equation*}
The preference partition of $\pi$ is $\{\{1\}, \{2\}, \{3,8\}, 
\{4,6,7\}, \{5\}, \{9\}\}$, and $K(\pi)=\{1,4,6,7\}$. One sees that 
$K(\hat \pi)=\{4,1,6,7\} = K(\pi)$, and $\pi$ and  $\hat \pi$  have the same reduced preference partition.   

Theorem~\ref{thm:map-rho} has some interesting consequences.  For example, we could take a set of  partitions $\mathcal{P}= (P_1, \dots, P_t)$ subject to certain conditions.  Two such examples are given below. 
\begin{corollary} 
   \begin{enumerate} 
   \item  Let $n \geq 2$. Statistics $(\lel(\pi), \mathrm{ones}(\pi))$ have a symmetric joint distribution over the set of classical parking functions $(\pi_1, \dots, \pi_n)$ with the condition that $\pi_2 \neq 1, \pi_1$. 
    \item Let $n \geq 3$. Statistics $(\lel(\pi), \mathrm{ones}(\pi))$ have a symmetric joint distribution over the set of classical parking functions $(\pi_1, \dots, \pi_n)$ with the condition that $\pi_2=\pi_3$ and $\pi_2  \neq 1 , \pi_1. $     
    \end{enumerate}
\end{corollary}
\begin{proof}
    The first one corresponds to the case that $1$ and $2$ are not in the same block of $\mathcal{P}$, while the second one is that $2$ and $3$ belong to the same block, which  does not contain $1$. 
\end{proof}

Clearly in the map $\rho$ one can replace vertex $p$ with the parent of any other vertex $i$ in the rooted tree, or equivalently, replace $\lel(\pi)$ with the multiplicity of $\pi_i$ for any fixed $i$ in classical parking functions.

\section{Extension to $(a, b)$-parking functions}
\label{sec:ab}

In this section we extend our results for parking functions to $(a, b)$-parking functions. Our proofs will again be based on refinements of Pollak's proof technique, where we investigate the individual parking statistics for each car the moment it is parked on the circle.

Let $\uR=(u_1,\dots,u_m)$ be a weakly increasing sequence of positive
integers. A \emph{$\uR$-parking function} is a sequence $(\pi_1,\dots,
\pi_m)$ of positive integers whose weakly increasing rearrangement
$\lambda_1\leq \lambda_2\leq \cdots\leq \lambda_m$ satisfies
$\lambda_i\leq u_i$. This in particular implies that $\boldsymbol{u}$-parking
functions, like parking functions, are permutation invariant. There is a similar interpretation
for such $\boldsymbol{u}$-parking functions in terms of the classical
parking scenario: One wishes to park $m$ cars on a street with $u_m$
spots, all spots begin unoccupied and the $i$-th largest occupied spot after parking has label $\leq u_{n+1-i}$ for $1\leq i\leq m$.

Fix positive integers $a,b$ and set
$\uR=(a,a+b,a+2b,\dots, a+(m-1)b)$, i.e., $u_i=a+(i-1)b$. A $\boldsymbol{u}$-parking
function of this form is referred to in \cite[Section 1.4.4]{Yan} as an \emph{$(a,
b)$-parking function} of length $m$. Denote the set of $(a, b)$-parking functions of length $m$ by $\PF(a, b, m)$. It coincides with parking functions $\PF(m, n)$ when $a=n-m+1$ and $b=1$. Note that for $(a,b)$-parking functions, if $\uR=(a, a+b, \dots, a+(n-1)b)$ and there are $m$ cars to park, then it is equivalent to parking $m$ cars with $\uR'=(a+(n-m)b, \dots, a+(n-1)b)$, which are counted by $\PF(a+(n-m)b, b, m)$.  Hence one only needs to consider the case $m=n$. The number of $(a, b)$-parking functions of length $m$ is well-known: $\vert \PF(a, b, m)\vert=a(a+mb)^{m-1}$.

As with parking functions, Pollak's circle argument applies to $(a,
b)$-parking functions: Let $G$ denote the group of all $m$-tuples
$(\pi_1,\dots,\pi_m)\in [a+mb]^m$ with componentwise addition modulo
$a+mb$. Let $H$ be the subgroup of $G$ generated by $(1,1,\dots,1)$. Then
every coset of $H$ contains exactly $a$ parking functions. Based on this combinatorial fact (for more details see Stanley \cite{Stanley1997}), we outline an effective method for generating a random $(a, b)$-parking function of length $m$. To select $\pi \in \PF(a, b, m)$ uniformly at random:
\begin{enumerate}
\item Pick an element $\pi \in (\mathbb{Z}/(a+mb)\mathbb{Z})^m$, where the equivalence class representatives are taken in $1, \dots, a+mb$.

\item For $k \in \{0, \dots, a+mb-1\}$, record $k$ if $\pi+k(1, \dots, 1)$ (modulo $a+mb$) is an $(a, b)$-parking function of length $m$, where $(1, \dots, 1)$ is a vector of length $m$. There should be exactly $a$ such $k$'s.

\item Pick one $k$ from (2) uniformly at random. Then $\pi+k(1, \dots, 1)$ is an $(a, b)$-parking function of length $m$ taken uniformly at random.
\end{enumerate}
The main takeaway from this procedure is that a random $(a, b)$-parking function of length $m$ could be generated by assigning $m$ cars independently on a circle of length $a+mb$ and then applying circular rotation.

The following lemma is an extension of Lemma \ref{simple} in Section \ref{sec:pf} from parking functions to $(a, b)$-parking functions.

\begin{lemma}\label{simple2}
We have
  \begin{equation*}
      \# \{\pi \in \PF(a, b, m): \pi_1=1\}=(a+b)(a+mb)^{m-2},   
\end{equation*}
which implies that
\begin{equation*}
      \# \{\pi \in \PF(a, b, m): \pi_1 \in \{1, 2, \dots, a\}\}=a(a+b)(a+mb)^{m-2}.    
\end{equation*}
\end{lemma}

\begin{proof}
The lemma may be proved by modifying the circular argument in the proof of Lemma \ref{simple}, but we present an alternative proof here. Note that by the definition of an $(a, b)$-parking function, changing $\pi_1 = 1$ to any value less than or equal to $a$ will still keep $\pi$ an $(a, b)$-parking function. The number of $(a, b)$-parking functions $\pi \in \PF(a, b, m)$ with $\pi_1 \in \{1, 2, \dots, a\}$ is thus $a$ times the number of parking functions $\pi \in \PF(a, b, m)$ with $\pi_1=1$ (actually, with $\pi_1=c$ where $c$ is any value $1\leq c\leq a$). We further note that the restricted parking condition on $\pi_1$ is equivalent to having $(\pi_2, \dots, \pi_m) \in \PF(a+b, b, m-1)$, for which there are $(a+b)(a+mb)^{m-2}$ possibilities.
\end{proof}

Previously, we adopted shorthand notations such as $\slev(\pi)$ and $\one(\pi)$ for different statistics on the parking function $\pi$. We will now provide refinements of our generating functions for parking functions and extend them to generic $(a, b)$-parking functions. Because there are more statistics to take into account, instead of inventing more shorthand notations, we will take a more direct approach. As in Section \ref{sec:BFS}, we denote by $\# i(\pi)$ the count of the number $i$ in the parking function $\pi$, so for example, $\# 1(\pi)$ coincides with $\one(\pi)$ in our earlier notation while $\#1(\pi) + \cdots +\#(n-m+1)(\pi)$ equals $\slev(\pi)$. When it is clear from the context, we drop the dependence on $\pi$ from the notation and simply write $\# i$.

Before we proceed with our enumerative results, let us describe a decomposition akin to Lemma~\ref{lm:decomp} for $(a,b)$-parking functions.

\begin{lemma}\label{lm:decomp-a-b}
Consider a function $\pi: [m] \to [a+(m-1)b]$. Fix the elements of $\pi$ that are equal to one of $1, 2, \dots, a$, and suppose that there are $s\geq 0$ such elements. Let the other elements be $\pi_{j_1} < \pi_{j_2} < \cdots < \pi_{j_{m-s}}$, and define a new function $\tilde{\pi}: [m-s] \to [m-1]$ by $\tilde{\pi}_i = \lceil \frac1{b} (\pi_{j_i} - a) \rceil$. Then $\pi$ is a parking function in $\PF(a,b,m)$ if and only if $\tilde{\pi}$ is a parking function in $\PF(m-s,m-1)$.

Every parking function $\pi \in \PF(a,b,m)$ can be uniquely reconstructed from the following information:
\begin{itemize}
    \item the subset $S \subseteq [m]$ that maps to $[a]$,
    \item the map from $S$ to $[a]$ obtained by restricting $\pi$ to $S$,
    \item the map $\tilde{\pi}$,
    \item the map $\alpha: [m-s] \to [b]$ that satisfies $\alpha_i \equiv \pi_{j_i} \mod b$.
\end{itemize}
\end{lemma}

\begin{proof}
    Note that $\tilde{\pi}$ maps to values between $1$ and $\lceil \frac1{b} (a+(m-1)b - a) \rceil = m-1$. By definition, $\pi$ is an $(a,b)$-parking function if and only if the $i$-th smallest element is at most $a+(i-1)b$, or equivalently if
    \begin{equation*}
        \big| \big\{k\,:\, \pi_k \leq a + (i-1)b \big\} \big|\geq i
    \end{equation*}
    for every $i \in [m]$. We can rewrite this as
    \begin{equation*}
    \big| \big\{k\,:\, \tfrac{1}{b}(\pi_k - a) \leq i-1 \big\} \big| = \big| \big\{k\,:\, \lceil \tfrac{1}{b}(\pi_k - a) \rceil \leq i-1 \big\} \big| \geq i.
    \end{equation*}
    All $k$ for which $\pi_k \leq a$ trivially belong to the set on the left. Since there are $s$ such elements by assumption, we thus find that this is equivalent to
    \begin{equation*}
    \big| \big\{h\,:\, \tilde{\pi}_h \leq i-1 \} \big| \geq  i-s,
    \end{equation*}
    and an index shift yields
    \begin{equation*}
    \big| \big\{h\,:\, \tilde{\pi}_h \leq i \} \big| \geq  i-s+1.
    \end{equation*}
    This is exactly the inequality that characterizes a parking function in $\PF(m-s,m-1)$, see~\eqref{pigeon}.

    To see that $\pi$ can be uniquely reconstructed from the given quadruple, note that $S$ and the restriction of $\pi$ to $S$ determine all values of $\pi$ that are less than or equal to $a$. The remaining values can be recovered from the identity $\pi_{j_i} = b(\tilde{\pi}_i - 1) + \alpha_i$.
    \end{proof}

Recall that $\PF(m, n)$ is a special case of $\PF(a, b, m)$ when $a=n-m+1$ and $b=1$. The following proposition is a refinement of as well as an extension to Corollary \ref{pf2}. 

\begin{proposition}\label{last1}
\begin{equation*}
\sum_{\pi \in \PF(a, b, m)}x_1^{\#1}\cdots x_a^{\#a} =(x_1+\cdots+x_a)(x_1+\cdots+x_a+mb)^{m-1}.
\end{equation*}
\end{proposition}

\begin{proof}
We use the decomposition in Lemma~\ref{lm:decomp-a-b}. If we fix $s$, there are $\binom{m}{s}$ possibilities for the set $S$. The restriction of $\pi$ to $S$ gives rise to a term $(x_1+\cdots+x_a)^s$ since every element of $S$ can map to any of the numbers $1,2,\ldots,a$. Finally, there are $s m^{m-s-1}$ possibilities for $\tilde{\pi}$ and $b^{m-s}$ possibilities for $\alpha$. Hence we have
\begin{align*}
\sum_{\pi \in \PF(a, b, m)}x_1^{\#1}\cdots x_a^{\#a} &= \sum_{s=1}^m \binom{m}{s} (x_1+\cdots+x_a)^s s m^{m-s-1} b^{m-s} \\
&= \sum_{s=1}^m \binom{m-1}{s-1} (x_1+\cdots+x_a)^s (mb)^{m-s} \\
&= (x_1+\cdots+x_a)(x_1+\cdots+x_a+mb)^{m-1},
\end{align*}
which completes the proof.
\end{proof}



\begin{remark}
Proposition 7.3 can also be obtained from  Corollary \ref{pf2} and the observation that any $\boldsymbol{u}$-parking function can be obtained uniquely by taking a classical parking function $(\pi_1, \dots, \pi_n)$ and then replacing each $i$ by an integer in the set $\{u_{i-1}+1, \dots, u_i\}$. In particular, the entry 1 can be replaced by any of $\{1, \dots, a\}$. This idea first appeared in \cite{PS1} to express the number of $\boldsymbol{u}$-parking functions in terms of the volume of the parking polytope.
\end{remark}


Recall that for classical parking functions $\pi \in \PF(n, n)$ (corresponding to $a=1$ and $b=1$), the statistics $\# 1(\pi)$ ($\one(\pi)$) and $\slev(\pi)$ coincide. The following proposition, which extends Proposition~\ref{last1} by also including $\lel(\pi)$, is thus also related to Corollary \ref{cor:pf1} (and Theorem \ref{thm:expl-formula}).

\begin{proposition}\label{last2}
\begin{align*}
\sum_{\pi \in \PF(a, b, m)}&x_1^{\#1}\cdots x_a^{\#a}y^{\lel(\pi)} \\
&= \sum_{j=1}^a x_j y (x_1 + \cdots + x_{j-1} + x_j y + x_{j+1} + \cdots + x_a + b) \\
&\qquad \times (x_1 + \cdots + x_{j-1} + x_j y + x_{j+1} + \cdots + x_a + mb)^{m-2} \\
&\qquad + b(m-1) (x_1 + \cdots + x_a) y (x_1 + \cdots + x_a + y + mb - 1)^{m-2}.
\end{align*}
\end{proposition}

\begin{proof}

We follow the lines of the previous proof, making use of the decomposition in Lemma~\ref{lm:decomp-a-b}. As in the proof of Theorem \ref{thm:expl-formula}, we distinguish two cases based on the preference of the first car.

\begin{itemize}
\item The first car prefers one of the spots $1,2,\ldots,a$: in this case, $1$ is an element of $S$, and there are $\binom{m-1}{s-1}$ possibilities to choose the rest of the set $S$. If $\pi_i = j$, then we need to include a factor $x_jy$ for the first car, and the remaining elements of $S$ give rise to the term $(x_1 + \cdots + x_{j-1} + x_j y + x_{j+1} + \cdots + x_a)^{s-1}$. Thus the contribution of this case to the generating function is
\begin{align*}
\sum_{j=1}^a &\sum_{s=1}^m \binom{m-1}{s-1} x_jy (x_1 + \cdots + x_{j-1} + x_j y + x_{j+1} + \cdots + x_a)^{s-1} s m^{m-s-1} b^{m-s} \\
&= \sum_{j=1}^a x_j y (x_1 + \cdots + x_{j-1} + x_j y + x_{j+1} + \cdots + x_a + b) \\
&\qquad \times (x_1 + \cdots + x_{j-1} + x_j y + x_{j+1} + \cdots + x_a + mb)^{m-2}. 
\end{align*}
\item The first car prefers a spot $> a$: in this case, $1$ does not belong to the set $S$, thus $j_1 = 1$. There are $\binom{m-1}{s}$ possibilities for the set $S$, which now plays no role concerning the statistic $\lel$. So we obtain the factor $\binom{m-1}{s}(x_1 + \cdots + x_a)^s$ as in Proposition~\ref{last1} from the set $S$ and the restriction of $\pi$ to $S$. For the remaining elements, we note that $\pi_{j_i} = \pi_{j_1} = \pi_1$ if and only if $\tilde{\pi}_i = \tilde{\pi}_1$ and $\alpha_i = \alpha_1$. Thus every element except for $\tilde{\pi}_1$ that contributes to $\lel(\tilde{\pi})$ contributes to $\lel(\pi)$ with probability $\frac{1}{b}$, and $\lel(\pi) - 1$ follows a binomial distribution given $\lel(\tilde{\pi})$. The contribution of this case to the generating function is thus
\begin{equation*}
\sum_{s=1}^{m-1} \binom{m-1}{s}(x_1 + \cdots + x_a)^s b^{m-s} \sum_{\tilde{\pi} \in \PF(m-s,m-1)} y \Big( 1 + \frac{y-1}{b} \Big)^{\lel(\tilde{\pi})-1}.
\end{equation*}
In view of Corollary~\ref{pf1}, this is equal to
\begin{equation*}
\sum_{s=1}^{m-1} \binom{m-1}{s}(x_1 + \cdots + x_a)^s b^{m-s} s y \Big( m + \frac{y-1}{b} \Big)^{m-s-1},
\end{equation*}
which simplifies to
\begin{equation*}
b(m-1) (x_1 + \cdots + x_a) y (x_1 + \cdots + x_a + y + mb - 1)^{m-2}.
\end{equation*}
\end{itemize}
Combining the two cases, we obtain the stated formula.
\end{proof}

If we specialize Proposition~\ref{last2} to $a = 1$, we obtain a generating function that is symmetric in $x$ and $y$, which implies that the statistics $\#1$ ($\one$) and $\lel$ are equidistributed for $(1,b)$-parking functions.

\begin{corollary}\label{last2-special}
\begin{align}\label{abgen1}
\sum_{\pi \in \PF(1, b, m)} x^{\#1(\pi)}y^{\lel(\pi)}=xy\left[(m-1)b(x+y+mb-1)^{m-2}+(xy+b)(xy+mb)^{m-2}\right].
\end{align}
\end{corollary}

Specializing the generating function in Corollary~\ref{last2-special} further by setting $x=1$ or $y=1$, we immediately obtain
the distributions of $\lel(\pi)$ and $\one(\pi)$. These are given in the following two corollaries.

\begin{corollary}
Taking $x=1$ in (\ref{abgen1}), we have
\begin{equation*}
\sum_{\pi \in \PF(1, b, m)} y^{\lel(\pi)}=y(y+mb)^{m-1}.
\end{equation*}
\end{corollary}

\begin{corollary}
Taking $y=1$ in (\ref{abgen1}), we have
\begin{equation*}
\sum_{\pi \in \PF(1, b, m)} x^{\one(\pi)}=x(x+mb)^{m-1}.
\end{equation*}
\end{corollary}

Asymptotic analysis for $(1, b)$-parking functions could be analogously carried out as for parking functions. The following proposition is stated without proof (cf. Proposition \ref{Poisson-CLT-1}).

\begin{proposition}
Let $b$ be a fixed positive integer. Take $m \to \infty$.
Consider the parking preference $\pi \in \PF(1, b, m)$ chosen uniformly at random. Then $\lel(\pi) - 1 \dconv \mathrm{Poisson}(1/b)$ and $\one(\pi) - 1 \dconv \mathrm{Poisson}(1/b)$, i.e., for every fixed nonnegative integer $j$, 
\begin{equation*}
\PR\left(\lel(\pi)=1+j \hspace{.1cm} \vert \hspace{.1cm} \pi \in \PF(1, b, m) \right) \sim \frac{(1/b)^j e^{-1/b}}{j!},
\end{equation*}
and
\begin{equation*}
\PR\left(\one(\pi)=1+j \hspace{.1cm} \vert \hspace{.1cm} \pi \in \PF(1, b, m) \right) \sim \frac{(1/b)^j e^{-1/b}}{j!}.
\end{equation*}
\end{proposition}

For our last result, we consider the special case $a = b = k$, where $k$ is any positive integer. In this case, we can modify the construction of Lemma~\ref{lm:decomp-a-b} slightly as follows. 

For a function $\pi: [m] \to [km]$, we define $\tilde{\pi}: [m] \to [m]$ by $\tilde{\pi}_i = \lceil \frac{\pi_i}{k} \rceil$, and $\alpha: [m] \to [k]$ is determined by $\alpha_i \equiv \pi_i \mod k$. It is clear that the pair $(\tilde{\pi},\alpha)$ determines $\pi$ uniquely. Moreover, an argument analogous to Lemma~\ref{lm:decomp-a-b} shows that $\pi$ is a $(k,k)$-parking function if and only if $\tilde{\pi}$ is a classical parking function, i.e., if $\tilde{\pi} \in \PF(m,m)$.

The final proposition is another variant of Corollary~\ref{cor:pf1}, now in the setting of $(k,k)$-parking functions. It exhibits an interesting symmetry that generalizes the symmetry between the $1$'s statistic $\one(\pi)$ and the leading elements statistic  $\lel(\pi)$ for classical parking functions $\pi \in \PF(n, n)$.
Let us say that $x \in [km]$ lies in the $i$-th \emph{block} if
$\lceil \frac{x}{k} \rceil =i$, meaning that 
$x$ belongs to the set $\{ik-k+1,ik-k+2,\ldots,ik\}$.
The block that contains the leading element of a parking function is called the \emph{leading block}. We write $\ell(\pi)$ for the index of the leading block. In other words, $\ell(\pi) = i$ if the leading element is one of $ik-k+1,ik-k+2,\ldots,ik$. We now consider a generating function that takes all elements in the first block and all elements in the leading block into account.





\begin{proposition} \label{prop:ab} 
We have
\begin{multline*}
\sum_{\pi \in \PF(k, k, m)}x_1^{\#1}\cdots x_k^{\#k}y_1^{\#(\ell(\pi)k-k+1)}\cdots y_k^{\#(\ell(\pi)k)}=\\
(m-1)(x_1+\cdots+x_k)(y_1+\cdots+y_k)(x_1+\cdots+x_k+y_1+\cdots+y_k+(m-1)k)^{m-2} \\
+(x_1y_1+\cdots+x_ky_k)(x_1y_1+\cdots+x_ky_k+k)(x_1y_1+\cdots+x_ky_k+mk)^{m-2}.
\end{multline*}
\end{proposition}

\begin{proof}
Consider the decomposition of $\pi$ into $\tilde{\pi}$ and $\alpha$ that was described above. Note that $\tilde{\pi}_1 = \lceil\frac{\pi_1}{k}\rceil = \ell(\pi)$.
Moreover, $\tilde{\pi}_j = 1$ if and only if $\pi_j \in \{1,2,\ldots,k\}$, and $\tilde{\pi}_j = \ell(\pi)$ if and only if $\pi_j \in \{\ell(\pi)k-k+1,\ldots,\ell(\pi)k\}$. Thus $1,2,\ldots,k$ are precisely the values that contribute to $\one(\tilde{\pi})$, and $\ell(\pi)k-k+1,\ldots,\ell(\pi)k$ are precisely the values that contribute to $\lel(\tilde{\pi})$.

By Corollary~\ref{cor:pf1}, we have
\begin{equation*}   
\sum_{\tilde{\pi} \in \PF(m,m)} x^{\one(\tilde{\pi})}y^{\lel(\tilde{\pi})} = 
xy (m-1)(x+y+m-1)^{m-2} + xy(xy + 1)(xy + m)^{m-2}.
\end{equation*}
As can be seen from the proof of Theorem~\ref{thm:expl-formula}, the first term corresponds to the case that car $1$ does not park in spot $1$, and the second term to the case that car $1$ parks in spot $1$. 
\begin{itemize}
\item In the former case, we have $\ell(\pi) \neq 1$. If $\alpha$ is taken uniformly at random, each element that contributes to $\one(\tilde{\pi})$ becomes an element that contributes to one of $\#1, \ldots, \#k$, each with probability $\frac1{k}$, depending on the corresponding value of $\alpha$. Likewise, each element that contributes to $\lel(\tilde{\pi})$ becomes an element that contributes to one of $\#(\ell(\pi)k-k+1), \ldots, \#(\ell(\pi)k)$, each with probability $\frac1{k}$. All these elements are independent of each other. On the level of generating functions, this can be modeled by replacing $x$ by $\frac1k(x_1+\cdots+x_k)$, replacing $y$ by $\frac1k(y_1+\cdots+y_k)$, and multiplying by $k^m$ (the number of possibilities for $\alpha$). Thus we obtain a total contribution of
\begin{equation*}
    (m-1)(x_1+\cdots+x_k)(y_1+\cdots+y_k)(x_1+\cdots+x_k+y_1+\cdots+y_k+(m-1)k)^{m-2}
\end{equation*}
from this case.
\item In the latter case, a similar argument applies, except that every element that contributes to $\one(\tilde{\pi})$ is also an element that contributes to $\lel(\tilde{\pi})$ and vice versa. We have $\ell(\pi) = 1$, so $\#1 = \#(\ell(\pi)k-k+1)$, and so on. Every element that contributes to $\one(\tilde{\pi}) = \lel(\tilde{\pi})$ now becomes an element that contributes to one of $\#1 = \#(\ell(\pi)k-k+1), \ldots, \#k = \#(\ell(\pi)k)$, each with probability $\frac{1}{k}$, and all elements are independent of each other. This can now be modeled by replacing $xy$ by $\frac1k(x_1y_1 + \cdots + x_ky_k)$ and multiplying by $k^m$, for a total contribution of
\begin{equation*}
    (x_1y_1+\cdots+x_ky_k)(x_1y_1+\cdots+x_ky_k+k)(x_1y_1+\cdots+x_ky_k+mk)^{m-2}.
\end{equation*}
\end{itemize}
Combining the two cases yields the desired result.
\end{proof}

\section{$(a,b)$-parking functions and multi-colored trees} \label{sec:rho-ab}


Corollary~\ref{last2-special}  and Proposition~\ref{prop:ab} demonstrate  symmetries among certain statistics of special $(a,b)$-parking functions. In this section we introduce the notion of $(a,b)$-colored trees, which are in one-to-one correspondence with  $(a,b)$-parking functions. Then we give various extensions of the involution $\rho$ of Section 6 to explain those symmetries.    

The $(a,b)$-colored trees  on $[n]_0$  
are rooted trees in $\calt(n)$ in which each edge from the root $0$ is labeled by one of $1, 2, \dots, a$, and each other edge is labeled by one of $1, 2, \dots, b$. 
We refer to the edge-labels as colors. 
See Figure~\ref{fig:ab-tree} for an illustration, where $a=3$, $b=2$, $n=8$, and the edge color is circled. 
 Such structures were first considered in \cite{Yan2001} as \emph{sequences of rooted $b$-forests} and  proved to be in bijection with $(a,b)$-parking functions under a bijection  induced by the modified BFS algorithm.

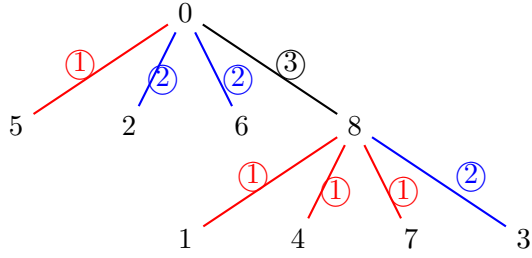
\begin{figure}[htbp]
    \centering
    \begin{tikzpicture}[scale=1]
    \node at (0,0) {0} [grow = down]
     child {node{5} edge from parent [color=red,thick] }
     child {node{2} edge from parent [color=blue,thick] } 
     child {node{6} edge from parent [color=blue,thick]} 
     child {node{8} edge from parent [thick] 
          child {node{1} edge from parent [color=red,thick]}
          child {node{4} edge from parent [color=red,thick]}
          child {node{7} edge from parent [color=red,thick]} 
          child {node{3} edge from parent [color=blue,thick]}} ;
    \node at (-1.4, -.7) {\red{\textcircled{1}}}; 
    \node at (-.3, -.9) {\blue{\textcircled{2}}}; 
    \node at (.7, -.9) {\blue{\textcircled{2}}}; 
    \node at (1.4, -.7) {{\textcircled{3}}}; 
    \node at (.9, -2.2) {\red{\textcircled{1}}}; 
    \node at (2, -2.4) {\red{\textcircled{1}}}; 
    \node at (2.9, -2.4)  {\red{\textcircled{1}}}; 
     \node at (3.8, -2.2) {\blue{\textcircled{2}}}; 
     
    \end{tikzpicture}
    \caption{An $(a,b)$-colored tree $T$  with $n=8$, $a=3$ and $b=2$. }
    \label{fig:ab-tree}
\end{figure}

First we describe the correspondence  from $(a,b)$-colored trees to $(a,b)$-parking functions as constructed in \cite{Yan2001}. We will again use  $\phi$ to represent this map, as the case with $a=b=1$ is exactly the one described in Section~\ref{sec:BFS}.
Given an $(a,b)$-colored tree $T$, the BFS order on its vertices is the same as the BFS order on  trees in $\calt(n)$, with the additional condition that vertices with the same predecessor are ordered in order of increasing edge colors, and, if the edges to their predecessor are of the same color,  in increasing order. For example, the $(a,b)$-colored tree in Figure~\ref{fig:ab-tree} has BFS order
\[
v_1 v_2 \cdots v_9 = 0, 5, 2, 6, 8, 1,  4, 7, 3.
\]
 Again, $\sigma_T^{-1}$ is the vertex ordering with $v_1=0$ removed, and $\sigma_T$ is the inverse permutation of $\sigma_T^{-1}$.

 For each vertex $v_i$ in the $(a,b)$-colored tree, the degree of $v_i$ is a sequence $(r_i^{(1)}, r_i^{(2)}, \dots)$  of length $a$ if $i=0$, and  $b$ otherwise, where the entry $r_i^{(j)}$ is the number of downward edges from $v_i$ of color $j$. The concatenation of the degrees of $v_1, \dots, v_{n}$ gives the specification $s(\pi)$, which, together with $\sigma_T$, yield the $(a,b)$-parking function $\pi=\phi(T)$. 

 For the above example, we have 
 $\sigma_T^{-1} = 5\ 2\ 6\ 8\ 1\ 4\ 7\ 3$, and hence,  $\sigma_T=5\ 2\ 8\ 6\ 1 \ 3\ 7\  4$. The specification is $s(\pi)=(1, 2, 1, 0,0,0,0,0,0,3,1,0,0,0,0,0,0,0,0)$. The non-decreasing rearrangement sequence associated with $s(\pi)$ is $1^1, 2^2, 3^1, 10^3, 11^1 = 1, 2,2, 3, 10, 10, 10, 11$. Therefore the $(a,b)$-parking function determined by $s(\pi)$ and $\sigma_T$ is $\pi=\phi(T)=(10, 2, 11, 10, 1, 2, 10, 3) \in \PF(3,2,8)$.   

 We remark that under the above bijection,  in the $(a,b)$-parking function,  $\pi=j$, or equivalently, car $i$ prefers spot $j$ if and only if 
 the following holds in the $(a,b)$-colored tree: 
 \begin{enumerate}
     \item if $1 \leq j \leq a$, then vertex $i$ is a child of vertex 0 
     with edge color $j$; 
     \item if $a+(k-1)b < j \leq a+kb$ for $k \geq 1$, then vertex $i$ is 
     a child of the $(k+1)$-th vertex $v_{k+1}$ with edge color $j-a \mod b$. 
 \end{enumerate}
It follows that  $(\#1(\pi), \#2(\pi), \dots, \#a(\pi))$ is exactly the degree of the root  $0$. 
Assume vertex $1$ is a child of vertex $p$, 
and the edge $(p,1)$ has color $j$.  
Then $\lel(\pi)=\# \pi_1(\pi)$ is the number of children of $p$ with edge color $j$. 
If $p \neq 0$, write $\pi_1  = a+kb +t$ for some integer $k \geq 0$ and $1\leq t \leq b$. Then $\deg(p)= (\#(a+kb+1)(\pi), \cdots, \#(a+kb+b)(\pi)).$ 

Next we construct  involution $\rho_{a,b}$ on the set of $(a,b)$-colored trees for the cases $a=1$ and $a=b=k$; both are extensions of the involution $\rho$ of Section 6 and  provide bijective proofs for the symmetries exhibited in Corollary~\ref{last2-special}  and Proposition~\ref{prop:ab}.

Definition of the bijection $\rho_{a,b}$ : 
 In an $(a,b)$-colored tree $T$ let the path from vertex $0$ to vertex $1$ be $0-q-\cdots - p -1$. Apply the following map if $p\neq 0$. 
 \begin{enumerate}
     \item \textbf{The case $\boldsymbol{a=1}$.}  Assume that the children of $0$ are $q, t_1, \dots, t_c$, which are all connected to $0$ with edges of color $1$.  Assume $(p,1)$ has color $j$. Among the children of $p$, assume the ones connected to $p$ with color $j$ are $1, s_1, \dots, s_d$.  Then $\rho_{1, b}(T)$ is obtained by
     \begin{enumerate}
         \item replacing the edges $(0, t_i)$ with $(p, t_i)$ of color $j$; 
         \item replacing the edges $(p, s_i)$  with $(0, s_i)$ of color 1; 
         \item keeping all other edges unchanged. 
     \end{enumerate}
     The map $\rho_{1,b}$ is a bijiection that swaps $\deg(0)$ with the $j$-th component of $\deg(p)$.
     
     \item \textbf{The case $\boldsymbol{a=b=k}$}. Assume that the children of $0$ are $q, t_1, \dots, t_c$, and the children of $p$ are $1, s_1, \dots, s_d$. Then $\rho_{k,k}(T)$ is obtained by 
     \begin{enumerate}
         \item replacing the edges $(0, t_i)$ with $(p, t_i)$ of the same color; 
         \item replacing the edges $(p, s_i)$ with $(0, s_i)$ of  the same color; 
         \item exchanging the color of edges $(0, q)$ and $(p,1)$; 
         \item keeping all other edges unchanged. 
     \end{enumerate}
       The map $\rho_{k,k}$ swaps $\deg(0)$ with $\deg(p)$, both of which are vectors with $k$ components.     
 \end{enumerate}

In a $(1,b)$-colored  (resp. $(k,k)$-colored) tree, 
for each vertex $t \neq 0, p$, 
the map $\rho_{1,b}$ (resp.~$\rho_{k,k}$)  preserves the set of children of $t$, as well as the color on the edges from $t$ to its children.

Letting $\hat \rho_{1,b} = \phi \circ \rho_{1,b}  \circ \phi^{-1}$, 
we get an involution on $P(1,b,n)$, which satisfies the 
 following  refinement of Corollary ~\ref{last2-special}. 

\begin{corollary}
    Fix a set partition $\mathcal{P}=\{P_1, P_2, \dots, P_t\}$ of $[n]$. Then $\lel(\pi)$ and $\one(\pi)$ have a symmetric joint distribution over all $(1,b)$-parking functions whose reduced preference partition is $\mathcal{P}$. 
\end{corollary}

Next we consider $(k,k)$-parking functions with positive integer $k$. 
Let $\pi \in \PF(k,k,n)$ and $\ell(\pi)= \lceil \frac{\pi_1}{k} \rceil$ be the index of the leading block.  
For each index $i$, let 
$S_\pi(i)=(B_\pi(ik-k+1), B_\pi(ik-k+2), \dots, B_\pi(ik))$, where $B_\pi(x)=\{j: \pi_j=x\}$ is the set of cars that prefer spot $x$. 
Let $\mathcal{O}(\pi)$ be the set 
$\{S_\pi(i): i \neq 1, \ell(\pi),  S_\pi(i)\neq (\emptyset, \dots \emptyset)\}$.  
Then the bijection $\hat \rho_{k,k} = \phi \circ \rho_{k,k} \circ \phi^{-1}$  preserves  $\mathcal{O}(\pi)$.

\begin{figure}[htbp]
 \centering
    \begin{tikzpicture}[scale=0.6]

   \node at (0, 3) {$\pi=(10, 3, 2, 9, 3, 1)$}; 

   \draw [->, thick] (0,1)--(0, 2.5); 
   \node at (0.3, 1.5) {$\phi$}; 
   
   \node at (10,3) {$\hat \pi = (9,5, 10, 1, 5, 2)$}; 

    \draw [->, thick] (10,1)--(10, 2.5); 
   \node at (10.3, 1.5) {$\phi$}; 
    \node at (0,0) {0} [grow = down]
     child {node{6} edge from parent [color=red,thick] 
        child {node{\textcolor{black}{2}} edge from parent [color=red,thick]}
        child {node{\textcolor{black}{5}} edge from parent [color=red,thick]
             child {node{\textcolor{black}{4}} edge from parent [color=red,thick]}
             child {node{\textcolor{black}{1}} edge from parent [color=blue,thick]}
              }
            }
     child {node{3} edge from parent [color=blue,thick] };

     \draw [<->, thick] (4,-2) -- (6,-2);
\node at (5, -1.5) {$\rho_{2,2}$}; 

    \node at (10,0) {0} [grow = down]
      child {node{4} edge from parent [color=red,thick] }
     child {node{6} edge from parent [color=blue,thick] 
        child {node{\textcolor{black}{2}} edge from parent [color=red,thick]}
        child {node{\textcolor{black}{5}} edge from parent [color=red,thick]
             child {node{\textcolor{black}{1}} edge from parent [color=red,thick]}
             child {node{\textcolor{black}{3}} edge from parent [color=blue,thick]}
              }
            };

    \end{tikzpicture}
    \caption{An example of a $(2,2)$-parking function $\pi$ and the corresponding $(2,2)$-colored tree $T$.   } 
    \label{fig:kk-PF}
\end{figure}
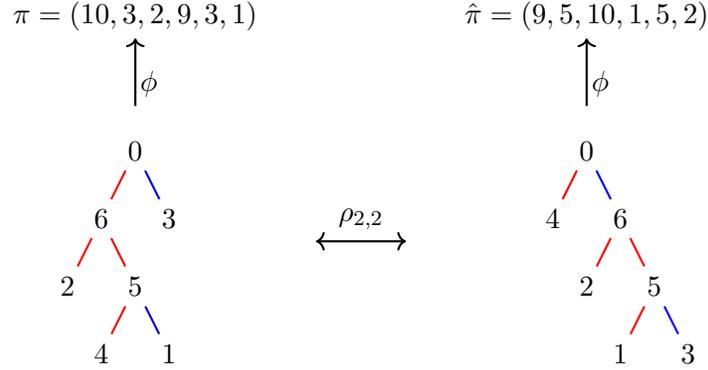

We explain the above concepts with
an example of $\pi=(10, 3, 2, 9, 3, 1) \in \PF(2,2,6)$, which corresponds to the $(2,2)$-colored tree on the left of 
Figure~\ref{fig:kk-PF}. Here red represents color 1 and blue represents color 2. 
Applying $\rho_{2,2}$ yields the 
$(2,2)$-colored tree on the right, which in turn gives 
$\hat \pi = (9, 5, 10, 1, 5, 2) \in \PF(2,2,6)$.  
The table below shows $S(i)$ for $\pi$ and $\hat \pi$. 

\begin{center} 
\begin{tabular}{|c|c|c|c|c|c|c|}
\hline
  & $S(1)$  & $S(2)$ & $S(3)$ & $S(4)$ & $S(5)$ & $S(6)$ \\
\hline
$\pi$  & $(\{6\}, \{3\}) $ & $(\{2,5\}, \emptyset)$ & $(\emptyset, \emptyset)$  & $(\emptyset, \emptyset)$  & $(\{4\}, \{1\})$ & $(\emptyset,\emptyset)$ \\
\hline
$\hat \pi$  & $(\{4\}, \{6\})$ & $(\emptyset,\emptyset)$ & $(\{2,5\}, \emptyset)$  & $(\emptyset,\emptyset)$ & $(\{1\}, \{3\})$ & $(\emptyset,\emptyset)$  \\
\hline
\end{tabular}
\end{center} 

In this example,  $\ell(\pi)=\ell(\hat \pi)=5$. From the table both 
$\mathcal{O}(\pi)$ and $\mathcal{O}(\hat \pi)$ are $\{ (\{2,5\}, \emptyset) \}$, as expected.

Let $\mathcal{O}=\{S(1), \dots, S(t)\}$, where 
\begin{enumerate}
    \item each $S(i)$ is a vector of length $k$ whose entries are subsets of $[n]$; 
    \item  $S(i) \neq (\emptyset, \dots, \emptyset)$; and 
    \item all the non-empty entries of $S(1), \dots, S(t)$ are mutually disjoint, and their  union is a subset of $\{2, \dots, n\}.$
\end{enumerate}
The involution $\hat \rho_{k,k}$ leads to the following refinement of Proposition~\ref{prop:ab}. 
\begin{proposition}
    Fix a set $\mathcal{O}$ as above. Then $(\#1, \dots, \#k)$ and $(\#(\ell(\pi)k-k+1), \dots, \#(\ell(\pi) k) )$ have a symmetric joint distribution over all $(k,k)$-parking functions $\pi$ with $\mathcal{O}(\pi)=\mathcal{O}$. 
\end{proposition}

Consequently, the symmetry between $(\#1, \dots, \#k)$ and $(\#(\ell(\pi)k-k+1), \dots, \#(\ell(\pi) k) )$ still holds if we take the set $\mathcal{O}$ with some constraints on its elements.
The final corollary gives two such examples.

\begin{corollary} 
   \begin{enumerate} 
   \item   Let $n \geq 2$. Statistics $(\#1, \dots, \#k)$ and $(\#(\ell(\pi)k-k+1), \dots, \# (\ell(\pi) k) )$   have a symmetric joint distribution over the set
    $$\{ (\pi_1, \dots, \pi_n) \in \PF(k,k,n):  
   \pi_2 \equiv t \mod k  \text{ and }   
    \Big\lceil \frac{\pi_2}{k}\Big\rceil  \neq   1,  \ell(\pi)\},$$ where $t$ is a fixed integer in $[k]$.      
    \item Let $n \geq 3$. Statistics $(\#1, \dots, \#k)$ and $(\#(\ell(\pi)k-k+1), \dots, \#(\ell(\pi) k ))$   
    have a symmetric joint distribution over the set 
    \[
    \{ (\pi_1, \dots, \pi_n) \in \PF(k,k,n):  
     \Big\lceil \frac{\pi_2}{k} \Big\rceil = \Big\lceil \frac{\pi_3}{k} \Big\rceil
    \text{ and }   
    \Big\lceil \frac{\pi_2}{k}\Big\rceil  \neq   1,  \ell(\pi)
    \}.
    \]
   
    \end{enumerate}
\end{corollary}
\begin{proof}
    The first means that $2$ is in the $t$-th entry of a vector  in $\mathcal{O}$, and the second means both $2$ and $3$ appear in the same vector in $\mathcal{O}$. 
\end{proof}

\section*{Acknowledgements}

Stephan Wagner and Mei Yin benefited from participation in the Workshop on Analytic and Probabilistic Combinatorics at the Banff International Research Station in November 2022 and also from participation in the 34th International Conference on Probabilistic, Combinatorial and Asymptotic Methods for the Analysis of Algorithms in Taipei in June 2023.


\bibliographystyle{abbrv}
\bibliography{bibliography}

\end{document}